\documentclass[a4paper,10pt]{article}

\usepackage[T2A]{fontenc}

\usepackage[english]{babel}

\usepackage{graphicx}
\usepackage{amsthm, amsmath, amssymb, amsbsy, amsfonts, dsfont}
\usepackage{srcltx, color}

\def\e{\varepsilon}
\def\g{\gamma}
\def\G{\Gamma}
\def\l{\lambda}

\def\a{\alpha}
\def\b{\beta}

\def\d{\delta}
\def\L{\Lambda}

\def\vp{\varphi}

\def\H{W_2}
\def\dH{\dot{W}_2}

\def\Hloc{W_{2,loc}}
\def\Hinf{W_\infty}
\def\dC{\mathfrak{C}}

\def\Th{\Theta}

\def\iu{\mathrm{i}}

\def\Op{\mathcal{H}}

\def\cL{\mathcal{L}}
\def\cS{\mathcal{S}}

\def\H{W_2}

\def\Hloc{W_{2,loc}}
\def\Hinf{W_\infty}

\def\cA{\mathcal{A}}
\def\cR{\mathcal{R}}

\def\cC{\mathcal{C}}
\def\cI{\mathcal{I}}
\def\cP{\mathcal{P}}
\def\cU{\mathcal{U}}
\def\cQ{\mathcal{E}}
\def\cV{\mathcal{V}}
\def\cL{\mathcal{L}}

\def\rA{\mathrm{A}}
\def\rB{\mathrm{B}}
\def\rU{\mathrm{U}}
\def\rI{\mathrm{I}}

\def\rS{\mathrm{S}}
\def\rY{\mathrm{Y}}
\def\rP{\mathrm{P}}
\def\rQ{\mathrm{Q}}
\def\rK{\mathrm{K}}

\def\rf{\mathrm{g}}
\def\ed{\mathrm{e}}
\def\ra{\mathrm{a}}
\def\rb{\mathrm{b}}

\def\rT{\mathrm{T}}
\def\rL{\mathrm{L}}
\def\rc{\mathrm{c}}
\def\rR{\mathrm{R}}
\def\rZ{\mathrm{Z}}
\def\Ups{\Upsilon}
\def\ry{\mathrm{y}}
\def\rh{\mathrm{h}}
\def\rG{\mathrm{G}}

\def\rx{\mathrm{x}}
\def\rX{\mathrm{X}}
\def\rE{\mathrm{E}}

\def\sp{\mathsf{p}}
\def\sq{\mathsf{q}}
\DeclareMathOperator{\rank}{rank}

\DeclareMathOperator{\IM}{Im}
\DeclareMathOperator{\diag}{diag}

\numberwithin{equation}{section}

\newtheorem{theorem}{Theorem}[section]
\newtheorem{lemma}{Lemma}[section]
\newtheorem{remark}{Remark}[section]

\begin{document}

\allowdisplaybreaks

\title{Analyticity of resolvents of elliptic operators on quantum graphs with small edges}

\date{\empty}

\author
{D.I. Borisov}

\maketitle

\begin{center}
{Institute of Mathematics, Ufa Federal Research Center, Russian Academy of Sciences, Ufa, Russia,
\\
\&
\\
Bashkir State
University,
Ufa, Russia, \\
\&
\\
University of Hradec Kr\'alov\'e, Hradec Kr\'alov\'e, Czech Republic
 \\
{\tt borisovdi@yandex.ru}}
\end{center}

\begin{abstract}
We consider an arbitrary metric graph, to which we glue another graph with edges of lengths proportional to $\e$, where $\e$ is a small positive parameter. On such graph, we consider a general self-adjoint second order differential operator $\Op_\e$ with varying coefficients subject to general vertex conditions; all coefficients in differential expression and vertex conditions are supposed to be analytic
in $\e$. We introduce a special operator on a certain graph obtained
by rescaling the aforementioned small edges and assume that it has no embedded eigenvalues at the threshold of its essential spectrum. Under such assumption, we show that certain parts of the resolvent of $\Op_\e$ are analytic in $\e$. This allows us to represent the resolvent of $\Op_\e$ by an uniformly converging Taylor-like series and its partial sums can be used for approximating the resolvent up to an arbitrary power of $\e$. In particular, the zero-order approximation reproduces recent convergence results by G. Berkolaiko, Yu. Latushkin, S. Sukhtaiev and by C. Cacciapuoti, but we additionally show that next-to-leading terms in $\e$-expansions of the coefficients in the differential expression and vertex conditions can contribute to the limiting operator producing the Robin part at the vertices, to which small edges are incident. We also discuss possible generalizations of our model including both the cases of a more general geometry of the small parts of the graph and a non-analytic $\e$-dependence of the coefficients in the differential expression and vertex conditions.
\end{abstract}

\section{Introduction}

The theory of quantum graphs or, in other words, the spectral theory of differential operators on graphs is actively developing nowadays and a rather new direction is devoted to studying quantum graphs with small edges. One of early studied problems came from the homogenization theory and it was devoted to models of woven membrane kind, see \cite{Pok}, \cite{Zh}. The well-known result states that an elliptic operator on a graph looking as a net with small edges converge to a two-dimensional operator on an Euclidean domain approximated in a natural sense by such graph.

Another situation, which attracted quite a lot of attention, concerned the graphs with finitely many small edges. As a rather early publication on this subject we mention paper \cite{CET}, in which it was shown that a general vertex condition in a graph can be approximated in the norm resolvent sense via ornamenting by small edges supporting a magnetic field and with delta-coupling at the vertices.

The case of a general graph with finitely small edges was addressed in \cite{BLS-AdvMath2019}. On such graph, a Schr\"odinger operator subject to general vertex conditions was considered. The main result stated that under a certain non-resonance condition, see Condition~3.2 in \cite{BLS-AdvMath2019}, the considered operator converged in the norm resolvent sense to a Schr\"odinger operator on the limiting graph without small edges, which were replaced by certain 
conditions at the limiting vertices. The vertex conditions in \cite{BLS-AdvMath2019} were described in terms of sympletic geometry and Lagrangian planes; the limiting vertex conditions were also formultated in the same terms. A similar study was made in \cite{Ital}, where there was considered a star graph, in which the vertex was replaced by a small rescaled finite graph. The operator on this perturbed graph was designed so that after the rescaling the small graph back to a finite size, the differential expression and vertex conditions turned out to be independent of a small parameter characterizing the size of the small graph. The main result of \cite{Ital} provided the leading terms in the asymptotics for the perturbed resolvent and the estimates for the error term. The limiting condition at the vertex of the limiting star graph contained only the Neumann and Dirichlet parts.

The above described results motivated further studies of the graphs with small edges
aimed on better understanding and describing the behavior of the resolvents. A natural question, apart of determining the limiting operator, is to find several next terms in the asymptotics for the resolvent or even its complete asymptotic expansion. This question was addressed in few recent papers~\cite{Bor1},~\cite{Bor2},~\cite{Bor3} for some toy models represented by very simple graphs. A main feature of the considered models was the possibility of singular dependence of the coefficients in the differential expression on a small parameter governing the lengths of small edges. The main results obtained for these models stated that the resolvents were analytic
in certain sense with respect to the small parameter and in the limit, the condition at the vertex to which the small edges shrank, could involve a Robin part. This Robin part was generated by the next-to-leading terms in the expansions of the coefficients of the differential expression with respect to the small parameter. For the resolvents, several first terms in their asymptotic expansions were found. The analytic
dependence of the resolvent on the small parameter was in some sense surprising since small edges is an example of a singular perturbation and usually, while dealing with singular perturbation, one can hope only to find an asymptotic expansion for the resolvent but not to prove that this expansion is a converging Taylor series.

The present work is motivated by papers \cite{BLS-AdvMath2019}, \cite{Ital}. We consider a general graph $\G$, to which a small graph $\g_\e$ is glued. This small graph is obtained by rescaling a given graph $\g$ into $\e^{-1}$ times, where $\e$ is a small parameter, see Figures~\ref{fig1},~\ref{fig2}. On this perturbed graph with small edges we define a second order scalar differential operator $\Op_\e$ depending on $\e$. The differential expression is of a general form involving first order terms and varying coefficients. The vertex conditions are also of a general matrix form with arbitrary not necessarily unitary matrix coefficients. All coefficients in the differential expression and the vertex conditions depend arbitrarily on $\e$. For the coefficients in the vertex conditions and in the differential expression on non-small edges in $\G$, this dependence is analytic
in $\e$, while the coefficients in the differential expression on the small edges in $\g_\e$ are meromorphic in $\e$. On the coefficients in the differential expression and the vertex conditions of $\Op_\e$ we impose conditions ensuring the self-adjointness of this operator. Then we attach finitely many leads to the graph $\g$ and denote such extended graph by $\g_\infty$, see Figure~\ref{fig2}. On the graph $\g_\infty$ we consider a self-adjoint operator $\Op_\infty$ with differential expression and vertex conditions generated in a certain way by those of the operator $\Op_\e$. Our main assumption is that the operator $\Op_\infty$ has no embedded eigenvalues at the bottom of its essential spectrum; this does not exclude the presence of possible virtual levels at the same bottom. This assumption is of a similar nature as Non-resonance condition in work \cite{BLS-AdvMath2019}. Under this condition, our main result states that the resolvent of the operator $\Op_\e$ is in some sense analytic
in $\e$. Namely, we consider the restrictions of this resolvent to the graphs $\G$ and $\g_\e$ and the latter restriction then is sandwiched between the rescaling operators; such parts of the resolvent turn out to be analytic in $\e$. As a corollary, this provides a Taylor-like series for the resolvent. All these results are established for the resolvent of $\Op_\e$ regarded not only as operator acting in $L_2$, but also as acting from $L_2$ into the Sobolev spaces $W_2^2$ and into the spaces of continuous functions $C^2$.

The mentioned Taylor-like series for the resolvent allows us to find a limiting operator on the graph $\G$, the resolvent of which approximates the resolvent of $\Op_\e$ in the sense of the above norms. This is a convergence result as in \cite{BLS-AdvMath2019} and \cite{Ital} but for our general model. Similar to the toy models in~\cite{Bor1},~\cite{Bor2},~\cite{Bor3}, our general model also has a special feature related with the limiting vertex condition at $M_0$. Namely, we show that the next-to-leading terms in the $\e$-expansions of the coefficients in the differential expression and vertex conditions of $\Op_\e$ generate a Robin part in the limiting vertex conditions at a vertex $M_0$, through which the graph $\g_\e$ is glued to the graph $\G$. Such case was excluded by the formulation of the problem in \cite{Ital}. In \cite{BLS-AdvMath2019}, the Robin part could appear in the limit, but only due to
the fact that, in our terms, the conditions at the vertices of the small edges were independent of $\e$.

Finally, we also conclude that despite the small edges are a singular perturbation, which can not be treated by means of the regular perturbation theory \cite{Kato},
the resolvent of the operator $\Op_\e$ depends analytically
in $\e$ and behaves similar to the situation, when the edges vary but do not vanish. The case with non-vanishing edges was considered in \cite{BK} and the main result stated that the resolvent of the Schr\"odinger operator on such graphs depended analytically on the edges lengths and the matrices involved in the vertex conditions.

\begin{figure}
\begin{center}
\includegraphics[scale=0.7]{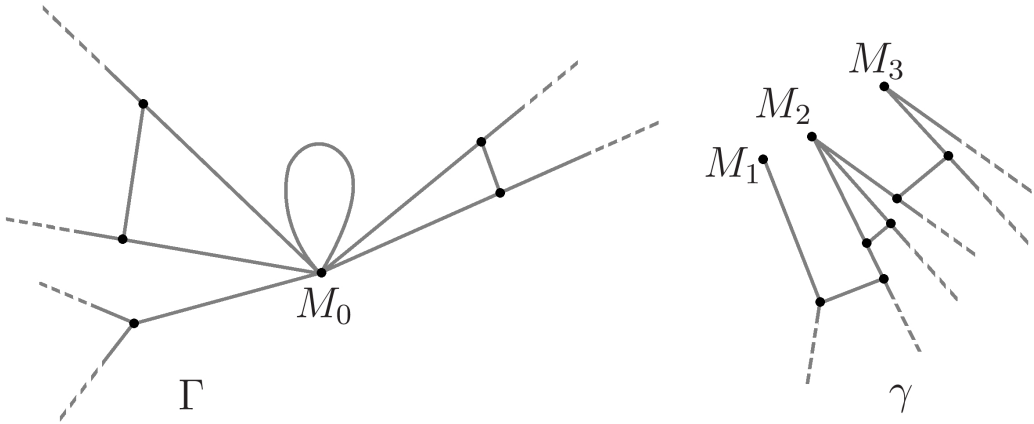}
\end{center}
\caption{{\small Examples of initial graphs $\G$ and $\g$. The edges incident to $M_0$
are split into three groups ($n=3$).}}\label{fig1}
\end{figure}

In conclusion of this section, let us briefly describe the structure of the paper. In the next section we formulate the problem, our main assumptions and the main results. In the third section we discuss how far our results can be extended to the case, when several small graphs are glued to the graph $\G$, see Figure~\ref{fig3}, including the case of the dependence on several independent small parameters. And we also discuss the case, when the assumed analyticity
in the small parameter for the coefficients in the differential expression and vertex conditions is replaced by the infinite differentiability in $\e$ or just by the existence of power asymptotic expansions.
In the fourth section we discuss the self-adjointness of the operator we deal with. In the fifth section we introduce and study some auxiliary operator on $\g$, which is employed then in proving the main results in the sixth section.

\section{Problem and main results}

\subsection{Problem}\label{sec2.1}

Let $\G$ be a finite metric graph having no isolated vertices. The lengths of its edges can be both finite or infinite. We fix arbitrarily a vertex in the graph $\G$ and denote it by $M_0$. The degree of this vertex is finite, that is, there exist finitely many edges $\ed_i$, $i=1,\ldots,d_0$ incident to $M_0$. Each loop in the family $\{\ed_i\}_{i=1,\ldots,d_0}$ is counted twice. By $\g$ we denote one more graph, which is supposed to be finite, metric, having only edges of finite lengths and no isolated vertices or edges, see Figure~\ref{fig1}.

\begin{figure}
\begin{center}
\includegraphics[scale=0.7]{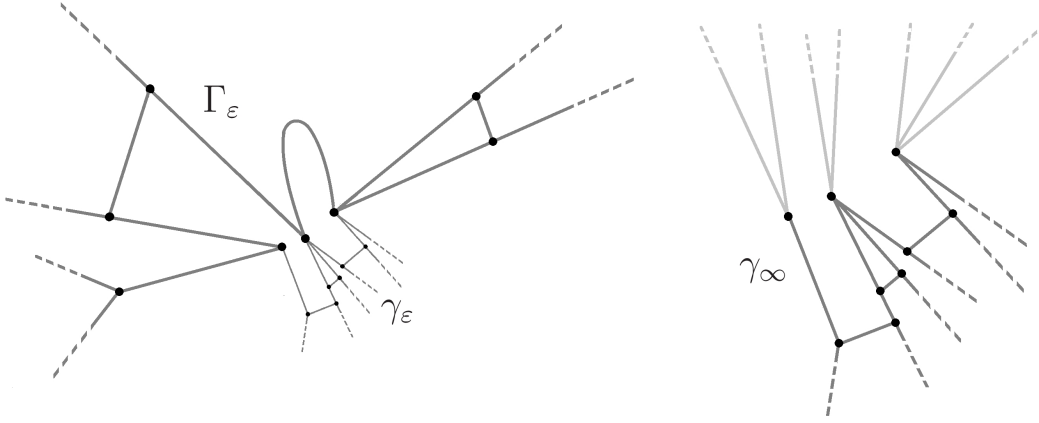}
\end{center}
\caption{{\small Graph $\G_\e$ with a glued small graph and graph $\g_\infty$. The small edges in the graph $\G_\e$ are indicated by thin lines. The leads in the graph $\g_\infty$ are of light gray color.}}
\label{fig2}
\end{figure}

Letting $\e$ to be a small positive parameter, we scale each edge of the graph $\g$ by $\e$ and the resulting graph is denoted by $\g_\e$. More precisely, the vertices of the graph $\g_\e$ are the same as those of $\g$, while each edge $\ed\in\g$ is replaced by the edge of the length $\e|\ed|$.

We chose arbitrary vertices $M_j$, $j=1,\ldots,n$, in the graph $\g_\e$ and then partition the aforementioned edges $\ed_i$, $i=1,\ldots,d_0$, in the graph $\G$ into $n$ non-empty groups $\{\ed_i\}_{i\in J_j}$, $j=1,\ldots,n$, where $J_j$ are some non-empty disjoint sets of indices and $\bigcup\limits_{j=1}^{n} J_j=\{1,\ldots,d_0\}$, $1\leqslant n\leqslant d_0$. Then
we replace the vertex $M_0$ in the graph $\G$ by its $n$ copies, one copy for each group $\{\ed_i\}_{i\in J_j}$. Finally, we introduce a graph, which will be the main object of our study, as the union of the graphs $\G$ and $\g_\e$, where the vertices $M_j$, $j=1,\ldots,n$, are identified with the aforementioned copies of the vertex $M_0$ in the graph $\G$, see Figure~\ref{fig2}. After the described gluing of the graphs $\G$ and $\g_\e$, the edges $e_i$, $i=1,\ldots,d_0$, are no longer connected at the vertex $M_0$ but instead, each group $\{\ed_i\}_{i\in J_j}$ is incident to the vertex $M_j$ in the graph $\g_\e$.
Hereafter, we shall often identify the graphs $\G$ and $\g_\e$ with the corresponding subgraphs in $\G_\e$. In particular, in the sense of this identification, each function defined on $\G_\e$ is also supposed to be defined on $\G$ and $\g_\e$ and vice versa.
On each edge in $\G$ and $\g$ we fix a direction and consequently, a variable on it. Then the chosen direction is naturally transferred to the graph $\G_\e$.

The main object of our study is an unbounded operator $\Op_\e$ in $L_2(\G_\e)$ with a differential expression $\hat{\Op}(\e)$, the action of which is defined as
\begin{equation}\label{2.1}
\hat{\Op}(\e)u:=-\frac{d\ }{dx} p_\e\frac{d u}{dx} + \iu \left(
\frac{d\ }{dx}(q_\e u) + q_\e\frac{du}{dx}
\right)+ V_\e u,
\end{equation}
where $\iu$ is the imaginary unit and the coefficients are 
\begin{equation}\label{2.1a}
\begin{gathered}
p_\e:=\left\{
\begin{aligned}
&p_\G(\,\cdot\,,\e) &&\text{on}\quad \G,
\\
\cS_\e &p_\g(\,\cdot\,,\e) &&\text{on}\quad\g_\e,
\end{aligned}\right.
\qquad
q_\e:=\left\{
\begin{aligned}
&q_\G(\,\cdot\,,\e) &&\text{on}\quad \G,
\\
\e^{-1}&\cS_\e q_\g(\,\cdot\,,\e) \quad&&\text{on} \quad \g_\e,
\end{aligned}\right.
\\
V_\e:=\left\{
\begin{aligned}
&V_\G(\,\cdot\,,\e) &&\text{on}\quad \G,
\\
\e^{-2}&\cS_\e V_\g(\,\cdot\,,\e) &&\text{on}\quad\g_\e.
\end{aligned}\right.
\end{gathered}
\end{equation}
Here $p_\G=p_\G(\,\cdot\,,\e)\in \Hinf^1(\G)$, $q_\G=q_\G(\,\cdot\,,\e)\in \Hinf^1(\G)$, $V_\G=V_\G(\,\cdot\,,\e)\in L_2(\G)$ and $p_\g=p_\g(\,\cdot\,,\e)\in\Hinf^1(\g)$, $q_\g=q_\g(\,\cdot\,,\e)\in\Hinf^1(\g)$, $V_\g=V_\g(\,\cdot\,,\e)\in L_2(\g)$ are some real   functions defined on the graphs $\G$ and
$V_\g$ 
and these functions are analytic
in $\e$ in the sense of the 
norms in the mentioned spaces.
The symbol $\cS_\e$ stands for a linear operator mapping $L_2(\g)$ onto $L_2(\g_\e)$ by the rule
\begin{equation*}
(\cS_\e u)(x):=u\left(\frac{x}{\e}\right)\quad\text{as}\quad x\in\ed_\e
\end{equation*}
on each edge $\ed_\e$ in the graph $\g_\e$.
 We also assume an uniform ellipticity condition for the expression $\hat{\Op}(\e)$, namely, the existence of a fixed constant $c_\Op>0$ independent of $x\in\G$, $\xi\in\g$ such that
\begin{equation*}
p_\G(x,0)\geqslant c_\Op\quad\text{a.e. on}\quad\G,\qquad
p_\g(\xi,0)\geqslant c_\Op\quad\text{a.e. on}\quad\g
\end{equation*}
Thanks to the assumed analyticity
in $\e$, this condition implies the same condition for $p_\G(x,\e)$ and $p_\g(x,\e)$ for all sufficiently small $\e$ with the constant $c$ replaced by $c/2$.

The vertex conditions on the graph $\G_\e$ are imposed as follows. Let $M$ be an arbitrary vertex in $\G_\e$ of a degree $d(M)>0$
and $\ed_i(M)$, $i=1,\ldots,d(M)$, be the edges incident to this vertex and $u_i:=u\big|_{\ed_i(M)}$, $i=1,\ldots,d(M)$, be the restrictions of a function $u$ to the edges $\ed_i(M)$. We introduce two $d(M)$-dimensional vectors
\begin{equation}\label{2.21}
\cU_M(u):=
\begin{pmatrix}
u_1(M)
\\
\vdots
\\
u_{d(M)}(M)
\end{pmatrix},\qquad \cU_M'(u):=
\begin{pmatrix}
\frac{d u_1}{dx_1}(M)
\\
\vdots
\\
\frac{d u_{d(M)}}{dx_{d(M)}}(M)
\end{pmatrix},
\end{equation}
where $x_i$ is the variable on the edge $\ed_i$; this variable is chosen according the above discussed direction. At each vertex $M\in\G_\e$ of a degree $d(M)>0$ we impose a general vertex condition
\begin{equation}\label{2.9}
\rA_M(\e)\cU_M(u)+\rB_M(\e)\cU_M'(u)=0,
\end{equation}
where $\rA_M(\e)$ and $\rB_M(\e)$ are some matrices of size $d(M)\times d(M)$ analytic
in $\e$.

The matrices $\rA_M(\e)$ and $\rB_M(\e)$ in conditions (\ref{2.4a}) are defined non-uniquely up to the left multiplication by an arbitrary non-degenerate $d(M)\times d(M)$ matrix. In view of this fact, we denote $r(M):=\rank \rB_M(0)$ and without loss of generality we additionally assume that for each $M\in \G_\e$, the first $r(M)$ rows in the matrix $\rB_M(0)$ are linearly independent and the other rows are zero, while each row among last $d(M)-r(M)$ ones in the matrix $\rA_M(0)$ is non-zero.

The domain of the operator $\Op_\e$ consists of the functions in $\dH^2(\g_\e)$ satisfying the imposed vertex conditions; hereinafter, for an arbitrary graph we denote $\dH^j(\,\cdot\,):= \bigoplus\limits_{\ed\in\,\cdot\,} \H^j(\ed)$, $j=1,2$. The action of the operator $\Op_\e$ on such functions is defined by differential expression (\ref{2.1}).

 We want the introduced operator $\Op_\e$ to be self-adjoint and this is ensured by the two following assumptions. The first of them says that the matrix
\begin{equation}\label{2.4a}
\rA_M(\e)\Pi_M^{-1}(\e)\rB_M^*(\e)+\iu \rB_M(\e)\Pi_M^{-1}(\e) \Th_M(\e)\Pi_M^{-1}(\e)\rB_M^*(\e)
\end{equation}
is self-adjoint, where
\begin{equation}\label{2.4b}
\begin{aligned}
&\Pi_M(\e):=\diag\big\{\nu_i(M) p_\e\big|_{\ed_i(M)}(M)\big\}_{i=1,\ldots,d(M)},
\\
&
\Th_M(\e):=\diag\big\{\nu_i(M) q_\e\big|_{\ed_i(M)}(M)\big\}_{i=1,\ldots,d(M)},
\end{aligned}
\end{equation}
and $\ed_i(M)$ are the edges incident to the vertex $M$, while $\nu_i(M):=1$ if the inward direction on the edge $\ed_i(M)$ at the vertex $M$ coincides with the above chosen direction on this edge and $\nu_i(M):=-1$ otherwise.

 The second assumption is the rank condition
$\rank\big(\rA_M(0)\ \; \rB_M(0)\big)=d(M)$ imposed at each vertex $M\in\G_\e$. In view of the above conditions for the structure of matrices $\rA_M$ and $\rB_M$, this condition
is obviously equivalent to
\begin{equation}\label{2.9rank}
\rank\big(\rA_M (\e) \ \;\rB_M(\e)\big)=d(M)
\end{equation}
for all sufficiently small $\e$.

\subsection{Notation and assumption}

To formulate our main results, we need some additional notations. By $\g_\infty$ we denote a graph obtained by attaching leads (edges of infinite lengths) $\ed_i^\infty$, $i\in J_j$, $j=1,\ldots,n$, to each vertex $M_j$, $j=1,\ldots,n$, in the graph $\g$; the vertex $M_j$ serves as the origin for the attached edges $\ed_i$. The variable on the graph $\g_\infty$ is denoted by $\xi$.

 The introduced graph appears in a very natural way. Namely, we consider the subgraph $\g_\e$ in $\G_\e$ with the edges $\ed_i$ incident to the vertices $M_j$. Then we rescale this subgraph and the edges by means of $\e$ so that the graph $\g_\e$ turns into the graph $\g$. Under such rescaling, the edges $\ed_i$ become ones with lengths of order $O(\e^{-1})$ and in the limit $\e\to+0$, we just replace them by the leads $\ed_i^\infty$.

The described scaling corresponds to the change of the variables $\xi=x\e^{-1}$ and we need to see how the part of the operator $\Op_\e$ corresponding to $\g_\e$ is transformed under such change. Namely, we first
replace the coefficients of the differential expression $\hat{\Op}_\e$ on the edges $\ed_i$ by their values at $M_0$. Then we multiply the obtained differential expression by $\e^2$, pass to the variable $\xi$ and $\e=0$. This gives rise to the differential expression
\begin{equation}\label{2.7}
\begin{aligned}
&\hat{\Op}_\infty u:=-\frac{d\ }{d\xi}p_\g(\,\cdot\,,0)\frac{d u}{d\xi}
+\iu\left(\frac{d\ }{d\xi}(q_\g(\,\cdot\,,0)u)+q_\g(\,\cdot\,,0)\frac{du }{d\xi}\right)
+V_\g(\,\cdot\,,0)u\quad\text{on}\quad\g,
\\
&\hat{\Op}_\infty u:=-\sp_i(0)\frac{d^2 u}{d\xi^2} \quad\text{on}\quad \ed_i^\infty,\quad i\in J_j,\quad j=1,\ldots,n,\qquad
 \sp_i(\e):=p_\G\big|_{\ed_i}(M_0,\e).
\end{aligned}
\end{equation}
Then we do a similar rescaling in the vertex conditions, but we multiply them by
\begin{equation*}
\tilde{\rE}_M(\e):=
 \begin{pmatrix}
 \rI_{r(M)} & 0
 \\
 0 & \e^{-1} \rI_{d(M)-r(M)}
 \end{pmatrix}.
\end{equation*}
We see that
\begin{align}\label{2.60a}
&\e\tilde{\rE}_M(\e) \rA_M(\e)=\rA_M^{(0)}
+O(\e),
 \qquad
\tilde{\rE}_M(\e) \rB_M(\e)=\rB_M^{(0)}
+O(\e),\quad
\\
&\rA_M^{(0)}:=
\begin{pmatrix}
0
\\
\rA_M^-(0)
\end{pmatrix}, \qquad \rB_M^{(0)}:=
\begin{pmatrix}
\rB_M^+(0)
\\
\frac{d\rB_M^-}{d\e}(0)
\end{pmatrix}\label{2.11b}
\end{align}
and $\rA_M^-(\cdot)$ and $\rB_M^-(\cdot)$ are the matrices formed by last $d(M)-r(M)$ rows of respectively the matrices $\rA_M(\cdot)$ and $\rB_M(\cdot)$, while a matrix $\rB_M^+(\cdot)$ is formed by first $r(M)$ rows of the matrix $\rB_M$.
Passing then to the limit as $\e\to+0$, we end up with the following vertex conditions:
\begin{equation}\label{2.18}
\rA_M^{(0)}\cU_M(u)+\rB_M^{(0)}\cU_M'(u)=0\quad\text{at each}\quad M\in\g_\infty,
\end{equation}
where the vectors $\cU_M(u)$ and $\cU_M'(u)$ are introduced as in (\ref{2.21}) with the derivatives $\frac{d u_i}{dx_i}$ substituted by $\frac{d u_i}{d\xi_i}$.

The described procedure gives rise to an auxiliary operator $\Op_\infty$ on the graph $\g_\infty$ with differential expression (\ref{2.7}) subject to the vertex conditions (\ref{2.18}).
 The domain of the operator $\Op_\infty$ consists of the functions in $\dH^2(\g)$
satisfying the imposed vertex conditions. It will be shown later, see Section~\ref{sec6.1}, that the operator $\Op_\infty$ is self-adjoint.

Since the graph $\g$ is finite and all its edges are of finite lengths, the only edges of infinite length in the graph $\g_\infty$ are the leads $\ed_i^\infty$ attached to the vertices $M_j$, $j=1,\ldots,n$. The differential expression of the introduced operator $\Op_\infty$ on these leads is just the negative Laplacian multiplied by non-zero constants $\sp_i$, see (\ref{2.7}). This is why it is straightforward to confirm that the essential spectrum of the operator $\Op_\infty$ is the half-line $[0,+\infty)$.

The main condition we assume in the work is as follows.
\begin{enumerate}
\def\theenumi{(A)}
\item\label{C1} \textsl{The operator $\Op_\infty$ has no embedded eigenvalue at the bottom of its essential spectrum.}
\end{enumerate}

At the bottom of its essential spectrum, the operator $\Op_\infty$ can have virtual level, namely, there can be a bounded non-trivial solution $\psi\in\dH^2(\g)
\oplus \bigoplus\limits_{i\in J_j,\ j=1,\ldots,n} \Hloc^2(\ed_i^\infty)$ of the boundary value problem
\begin{equation}\label{2.20}
\hat{\Op}_\infty \psi=0\quad\text{on}\quad \g_\infty,\qquad \rA_M^{(0)} \cU_M(\psi)+\rB_M^{(0)}\cU'_M(\psi)=0\quad \text{at each}\quad M\in\g_\infty.
\end{equation}
Again in view of the definition of the differential expression $\hat{\Op}_\infty$ on the leads $\ed_i^\infty$, the aforementioned boundedness condition for $\psi$ is equivalent to the identities
$\psi=const$ on $\ed_i^\infty$, $i\in J_j$, $j=1,\ldots,n$,
where $const$ stands for some constants depending in general on the choice of the edge $\ed_i^\infty$.

Given a function $u$ defined and continuous
on the edges $\ed_i^\infty$ in the vicinity of the points $M_j$, we denote
\begin{equation*}
\cU_\g(u):=
\begin{pmatrix}
u\big|_{\ed_i^\infty}(M_j)
\end{pmatrix}_{i\in J_j,\ j=1,\ldots,n}
\end{equation*}
where $u\big|_{\ed_i^\infty}$ is the restriction of $u$ to the edge $\ed_i^\infty$.
In view of the said above, Condition~\ref{C1} is equivalent to the following: each non-trivial solution $\psi$ to problem (\ref{2.20}) satisfies
\begin{equation}\label{2.23}
\cU_\g(\psi)\ne0.
\end{equation}

Let $\psi^{(j)}$, $j=1,\ldots,k$, be linearly independent bounded non-trivial solutions to problem (\ref{2.20}) satisfying condition (\ref{2.23}). It is clear that $k\leqslant d_0$ since once $k>d_0$ it is possible to find a linear combination of the functions $\psi^{(j)}$, for which condition (\ref{2.23}) is violated. If the operator $\Op_\infty$ has no virtual level at the bottom of its essential spectrum, we let $k:=0$.

We denote:
$\varPsi^{(j)}:=\cU_\g(\psi^{(j)})$,
$j=1,\ldots,k$.
We choose the functions $\psi^{(j)}$ so that the associated vectors $\varPsi^{(j)}$ are orthonormalized in $\mathds{C}^{d_0}$. If $k<d_0$, we choose arbitrary vectors $\varPsi^{(j)}\in \mathds{C}^{d_0}$, $j=k+1,\ldots,d_0$, so that the vectors $\varPsi^{(j)}\in \mathds{C}^{d_0}$, $j=1,\ldots,d_0$, form an orthonormalized basis in $\mathds{C}^{d_0}$. Hence, the matrix
$\Psi:=
\begin{pmatrix}
\varPsi^{(1)} & \ldots &\varPsi^{(k)} & \varPsi^{(k+1)} & \ldots &\varPsi^{(d_0)}
\end{pmatrix}$
is unitary.

\subsection{Parts of the resolvent and main result}

Our main result concerns certain operators, which are naturally regarded as some parts of the resolvent $(\Op_\e-\l)^{-1}$. In this subsection we introduce these operators and their limits as $\e\to+0$ and we formulate our main result.

Let $\cP_{\G}: L_2(\G_\e)\to L_2(\G)$ and $\cP_{\g_\e}: L_2(\G_\e)\to L_2(\g_\e)$ be the operators of restriction to the graphs $\G$ and $\g_\e$, that is, on each $f\in L_2(\G_\e)$ they act as follows:
$\cP_{\G}f:=f\big|_{\G}$, $\cP_{\g_\e}f:=f\big|_{\g_\e}$. In the sense of the decomposition $L_2(\G_\e)=L_2(\G)\oplus L_2(\g_\e)$ these operators satisfy the identity
\begin{equation}\label{2.3}
\cP_{\G}\oplus\cP_{\g_\e}=\cI_{\G_\e},
\end{equation}
where $\cI_{\G_\e}$ is the identity mapping in $L_2(\G_\e)$.

For each $\l\in\mathds{C}\setminus\mathds{R}$ the resolvent $(\Op_\e-\l)^{-1}$ is well-defined.
The operators
\begin{equation*}
\cR_\G(\e,\l):=\cP_{\G}(\Op_\e-\l)^{-1}(\cI_{\G}\oplus
\cS_\e),\qquad
\cR_\g(\e,\l):=\cS_\e^{-1} \cP_{\g_\e}(\Op_\e-\l)^{-1}(\cI_{\G}\oplus \cS_\e)
\end{equation*}
are also well-defined, linear and are bounded as acting from $L_2(\G)\oplus L_2(\g)$ into $\dH^2(\G)$ and $\dH^2(\g)$; here the direct sum is understood in the sense of identity (\ref{2.3}).

 Let us clarify the action of the operators $\cR_\G$ and $\cR_\g$.
Given $f\in L_2(\G_\e)$ and $u_\e:=(\Op_\e-\l)^{-1}f$, we consider the restrictions of these functions to the subgraphs $\G$ and $\g_\e$. The restrictions to $\g_\e$ are additionally rescaled by means of the operator $\cS_\e$, that is, they are simply regarded as functions of the rescaled variable $\xi:=x\e^{-1}$ defined on $\g$. These restrictions, with the rescaling taken into account, are obviously $\cP_\G f$, $\cP_\G u_\e$ and $\cS_\e^{-1} \cP_{\g_\e}f$, $\cS_\e^{-1} \cP_{\g_\e}u_\e$. Then the operators $\cR_\G(\e,\l)$ and $\cR_\g(\e,\l)$ map a pair $(\cP_\G f,\,\cS_\e^{-1} \cP_{\g_\e}f)$ respectively into $\cP_\G u_\e$ and $\cS_\e^{-1} \cP_{\g_\e}u_\e$.
 We also observe an obvious identity implied immediately by the definition of the operators $\cR_\G(\e,\l)$ and $\cR_\g(\e,\l)$:
\begin{equation}\label{2.35}
(\Op_\e-\l)^{-1}=\big(\cR_\G(\e,\l) \oplus \cS_\e\cR_\g(\e,\l)\big)(\cP_\G\oplus \cS_\e^{-1}\cP_{\g_\e}).
\end{equation}

 Now we introduce an operator, which will serve as a limiting one for $\cR_\G(\e,\l)$. Since the operator $\cR_\G(\e,\l)$ represents the action of the resolvent $(\Op_\e-\l)^{-1}$ on the subgraph $\G$, its limit should be the resolvent $(\Op_0-\l)^{-1}$, where $\Op_0$ is an operator in $L_2(\G)$
with the differential expression $\hat{\Op}(0)$
subject to the vertex conditions
\begin{equation}\label{2.34}
\rA_{M}^{(0)}\cU_M(u)+\rB_{M}^{(0)}\cU_M'(u)=0\quad\text{at each vertex}\quad M\in \G.
\end{equation}
For the vertices $M\ne M_0$, the matrices in the above vertex condition are defined as
$\rA_M^{(0)}:=\rA_M(0)$, $\rB_M^{(0)}:=\rB_M(0)$; the domain of the operator $\Op_0$
is formed by the functions from $\dH^2(\G)$
satisfying vertex conditions (\ref{2.34}).

The structure of the matrices $\rA_{M_0}^{(0)}$, $\rB_{M_0}^{(0)}$ is more complicated and heuristically they can be found as follows. Given $f\in L_2(\G_\e)$, let $u_\e:=(\Op_\e-\l)^{-1}f$. Since the coefficients in the differential expression and boundary conditions of $\Op_\e$ depend analytically in $\e$ and the graph $\g_\e$ is also obtained via rescaling by means of $\e$, it is natural to expect that on $\G_\e$, we have two approximations
\begin{equation}\label{2.38}
u_\e(x)=u_0(x)+\ldots\quad\text{on}\quad \G,
\qquad
u_\e(x)=v_0(\xi) + \e v_1(\xi)+\ldots\quad\text{on}\quad \g_\e,\qquad \xi:=x\e^{-1},
\end{equation}
where $u_0$, $v_0$, $v_1$ are some functions.
Since these are approximations for the same function $u_\e$, which is continuously differentiable on the edges $\ed_i$, $i=1,\ldots,d_0$, we should obviously have
\begin{equation}\label{2.39}
\cU_{M_0}(u_0)=\cU_\g(v_0),\qquad \cU_{M_0}'(u_0)=
\begin{pmatrix}
v_1'\big|_{\ed_i^\infty}(M_j)
\end{pmatrix}_{i\in J_j,\ j=1,\ldots,n},\qquad
\begin{pmatrix}
v_0'\big|_{\ed_i^\infty}(M_j)
\end{pmatrix}_{i\in J_j,\ j=1,\ldots,n}=0.
\end{equation}
We substitute then (\ref{2.38}) into the equation $(\Op_\e-\l)u_\e=f$ considered on $\g_\e$, pass to the variables $\xi$, and equate the coefficients at the like powers of $\e$. This gives immediately the boundary value problem for $v_0$
\begin{equation}\label{2.40}
\hat{\Op}_\infty v_0=0\quad\text{on}\quad\g_\infty,\qquad
\rA_M^{(0)}\cU_M(v_0)+\rB_M^{(0)}\cU_M'(v_0)=0\quad\text{at}\quad M\in\g_\infty,
\end{equation}
and for $v_1$
\begin{equation}\label{2.42}
\begin{aligned}
&\hat{\Op}_\infty v_1= \frac{d\ }{d\xi}\frac{d p_\g}{d\e}(\,\cdot\,,0)\frac{d v_0}{d\xi}
-\iu\left(\frac{d\ }{d\xi}\frac{d q_\g}{d\e}(\,\cdot\,,0)v_0+\frac{d q_\g}{d\e}(\,\cdot\,,0)\frac{d v_0}{d\xi}\right)
-\frac{dV_\g^{(0)}}{d\e}(\,\cdot\,,0)v_0
\quad\text{on}\quad\g,
\\
&\rA_M^{(0)}\cU_M(v_1)+\rB_M^{(0)}\cU_M'(v_1)=
-\rA_M^{(1)}\cU_M(v_0)-\rB_M^{(1)}\cU_M'(v_0)
\quad\text{at}\quad M\in\g_\infty,
\end{aligned}
\end{equation}
where $\rA_M^{(1)}$ and $\rB_M^{(1)}$ are the matrices in the next terms in expansions (\ref{2.60a}); they read as
\begin{equation*}
\rA_M^{(1)}:=
\begin{pmatrix}
\rA_M^+(0)
\\
\frac{d\rA_M^-}{d\e}(0)
\end{pmatrix},\qquad \rB_M^{(1)}:=
\begin{pmatrix}
\frac{d\rB_M^+}{d\e}(0)
\\
\frac{1}{2}\frac{d^2\rB_M^-}{d\e^2}(0)
\end{pmatrix}
\end{equation*}

In view of the first and second conditions in (\ref{2.39}), the solution of problem (\ref{2.40}) reads as
\begin{equation}\label{2.43}
v_0 =\sum\limits_{i=1}^{k} c_i \psi^{(i)}, \quad \begin{pmatrix}
 c_1(f) &
 \ldots &
 c_k(f)
 \end{pmatrix}^t
 :=
 \begin{pmatrix}
 \varPsi^{(1)} & \ldots &\varPsi^{(k)}
 \end{pmatrix} \cU_{M_0}\big((\Op_\G-\l)^{-1}f\big).
\end{equation}
In order to ensure the solvability of problem (\ref{2.42}), we multiply the equation for $v_1$ by $\psi^{(i)}$, $i=1,\ldots,k$, and integrate twice by parts over $\g$. In view of formula (\ref{2.43}) and the second identity in (\ref{2.39}), such procedure gives certain solvability condition relating the vectors $\cU_{M_0}(u_0)$ and $\cU_{M_0}'(u_0)$. The relation obtained in this way is exactly the desired vertex condition for the operator $\Op_0$ at the vertex $M_0$.

The corresponding matrices $\rA_{M_0}^{(0)}$ and $\rB_{M_0}^{(0)}$ turn out to be of form
\begin{equation}\label{2.32}
\rA_{M_0}^{(0)}:=
\begin{pmatrix}
\rQ & 0
\\
0 & \rI_{d_0-k}
\end{pmatrix}\Psi^*+\iu\begin{pmatrix}
\rI_{k} & 0
\\
0 & 0
\end{pmatrix}\Psi^* \Th_{\G,M_0}(0),
\qquad
\rB_{M_0}^{(0)}:=-
\begin{pmatrix}
\rI_k & 0
\\
0 & 0
\end{pmatrix}\Psi^*\Pi_{\G,M_0}(0),
\end{equation}
where the symbol $0$ in the first row of the matrix $\rA_{M_0}$ stands for the zero matrix of the size $k\times(d_0-k)$, while in the second row the same symbol denotes the zero matrix of the size $(d_0-k)\times k$.
In the definition of the matrix $\rB_{M_0}$, the first matrix is of the size $d_0\times d_0$ and the symbols $0$ stands for the zero matrices of respectively the sizes $k\times(d_0-k)$, $(d_0-k)\times k$, and $(d_0-k)\times (d_0-k)$. The symbol $\rI_d$ denotes the unit matrix of size $d\times d$. The matrices $\Pi_{\G,M_0}$, $\Th_{\G,M_0}$ are defined as
\begin{equation*}
\Pi_{\G,M_0}(\e):=\diag\big\{\nu_i(M_0) p_\G\big|_{\ed_i}(M,\e)\big\}_{i=1,\ldots,d_0},\qquad
\Th_{\G,M_0}(\e):=\diag\big\{\nu_i(M_0) q_\G\big|_{\ed_i}(M,\e)\big\}_{i=1,\ldots,d_0},
\end{equation*}
where $\ed_i$ are the edges incident to the vertex $M_0$, while $\nu_i(M_0)$ is defined in the same way as in (\ref{2.4b}).

The matrix $\rQ$ is exactly one appearing via integrating by parts, while obtaining the aforementioned solvability condition of problem (\ref{2.42}): 
\begin{equation*}
\rQ:=
\begin{pmatrix}
Q^{(11)} &\ldots & Q^{(k1)}
\\
\vdots &&\vdots
\\
Q^{(1k)}&\ldots & Q^{(kk)}
\end{pmatrix},
\end{equation*}
with the entries
\begin{align}\label{2.31}
&
\begin{aligned}
Q^{(ij)}:=Q_\g^{(ij)}+\sum\limits_{M\in\g_\infty} Q_M^{(ij)},
\end{aligned}
\\
&
\begin{aligned}
Q_\g^{(ij)}:=&\left(\frac{d p_\g}{d\e}(\,\cdot\,,0)\frac{d\psi^{(i)}}{d\xi},\frac{d\psi^{(j)}}{d\xi}\right)_{L_2(\g)} +
\left(\frac{d\psi^{(i)}}{d\xi},\iu \frac{d q_\g}{d\e}(\,\cdot\,,0) \psi^{(j)}\right)_{L_2(\g)}
\\
&+
\left(\iu \frac{d q_\g}{d\e}(\,\cdot\,,0) \psi^{(i)},\frac{d\psi^{(j)}}{d\xi}\right)_{L_2(\g)}
+\left(\frac{d V_\g}{d\e}(\,\cdot\,,0)\psi^{(i)},\psi^{(j)}\right)_{L_2(\g)},
\end{aligned}\label{2.25a}
\\
&
Q_M^{(ij)}:=\big(
\cL_M(\psi^{(i)}),
\cU_M(\psi^{(j)})\big)_{\mathds{C}^{d(M)}}
-\frac{\iu}{2}\Big(
\cQ_M(\psi^{(i)}), \cV_M(\psi^{(j)})
\Big)_{\mathds{C}^{d(M)}},
\label{2.25b}
\\
&\cL_M(\psi^{(i)}):=\frac{d\Pi_{\g,M}}{d\e}(0) \cU_M'(\psi^{(i)})-
\iu \frac{d \Th_{\g,M}}{d\e}(0)\cU_M(\psi^{(i)})
+\big(\rU_M^{(0)}+\rI_{d(M)}\big)^{-1}
\rP_{M,\bot}^{(0)}\cQ_M(\psi^{(i)}),\label{2.25c}
\\
&
\cQ_M(\,\cdot\,):= 2\iu\big(\tilde{\rA}_M^{(0)}-\iu \tilde{\rB}_M^{(0)}\big)^{-1} \Big(\rA_M^{(1)}\cU_M(\,\cdot\,)+\rB_M^{(1)} \cU_M'(\,\cdot\,) \Big), \label{2.29}
\\
&\cV_M(\,\cdot\,):=\Pi_{\g,M}(0)
\cU_M'(\,\cdot\,)-\iu \Th_{\g,M}(0)
\cU_M(\,\cdot\,), \nonumber
\\
&
\begin{aligned}
\tilde{\rA}_M^{(0)}:=\rA_M^{(0)}+\iu \rB_M^{(0)} \Pi_{\g,M}^{-1}(0) \Th_{\g,M}(0),\qquad \tilde{\rB}_M^{(0)}:=\rB_M^{(0)} \Pi_{\g,M}^{-1}(0),
\end{aligned}\nonumber
\\
&\rU_M^{(0)}:=-\big(\tilde{\rA}_M^{(0)}-\iu\tilde{\rB}_M^{(0)}\big)^{-1}\big(\tilde{\rA}_M^{(0)}+\iu \tilde{\rB}_M^{(0)}\big),\label{2.30a}
\\
&
\begin{aligned}
&\Pi_{\g,M}(\e):=\diag\big\{\nu_i(M) p_\g \big|_{\ed_i(M)}(M,\e)\big\}_{i=1,\ldots,d(M)},
\\
&\Th_{\g,M}(\e):=\diag\big\{\nu_i(M) q_\g \big|_{\ed_i(M)}(M,\e)\big\}_{i=1,\ldots,d(M)},
\end{aligned}
\label{2.26a}
\end{align}
where $\ed_i(M)$ are the edges incident to the vertex $M$, the numbers $\nu_i(M)$ are defined as in (\ref{2.4b}), and while applying formula (\ref{2.26a}) to $M=M_j$, the functions $V_\g^{(l)}$ are supposed to be continued on $\ed_i^\infty$, $i\in J_j$, $j=1,\ldots,n$, as follows:
$p_\g(\,\cdot\,,\e)\equiv\sp_i(\e)$,
$q_\g(\,\cdot\,,\e)\equiv\e \sq_i(\e)$, $\sq_i(\e):=q_\G\big|_{\ed_i}(M_0,\e)$.
We shall show that the matrix $\rU_M^{(0)}$ defined in (\ref{2.30a}) is unitary, see Lemma~\ref{lm7.1}.
 By $\rP_M^{(0)}$ we denote the projector in $\mathds{C}^{d(M)}$ onto the eigenspace of the matrix $\rU_M^{(0)}$ associated with the eigenvalue $-1$, and we let $\rP_{M,\bot}^{(0)}:=\rI_{d(M)}-\rP_M^{(0)}$.


 The second identity in (\ref{2.38})
means that the function $v_0$ should represent the limit of the operator $\cR_\g$ as $\e\to+0$.
 Namely, in view of formula (\ref{2.43}), we introduce one more operator $\cR_\g^{(0)}(\l): L_2(\G)\to \dH^2(\g)$ acting on each
 $f\in L_2(\G)$ by the rule:
\begin{gather*}
\cR_\g^{(0)}(\l)f:=\sum\limits_{i=1}^{k} c_i(f)\psi^{(i)},
\\
 \begin{pmatrix}
 c_1(f) & \ldots &
 c_k(f)
 \end{pmatrix}^t
 := \begin{pmatrix}
 \varPsi^{(1)} & \ldots &\varPsi^{(k)}
 \end{pmatrix}
 \cU_{M_0}\big((\Op_0-\l)^{-1}f\big).
\end{gather*}
We shall show later that this operator is indeed the limiting one for $\cR_\g(\e,\l)$.

We introduce the following spaces of continuous functions on graphs:
\begin{align*}
&\dC(\,\cdot\,):=\dC^0(\,\cdot\,):=\bigoplus\limits_{\ed\in\,\cdot\,} C(\overline{\ed})\cap L_\infty(\ed), \qquad
\dC^1(\,\cdot\,):=\bigoplus\limits_{\ed\in\,\cdot\,} C^1(\overline{\ed})\cap W_\infty^1(\ed),
\\
&\dC^2(\,\cdot\,):=\bigoplus\limits_{\ed\in\,\cdot\,} C^1(\overline{\ed})\cap W_\infty^2(\ed),
\qquad
\|u\|_{\dC^i(\overline{\ed})}:=\sum\limits_{\ed\in\cdot} \|u\|_{W_\infty^i(\ed)}.
\end{align*}
If the functions $p_\G$, $q_\G$, $V_\G$ and $p_\g$, $q_\g$, $V_\g$ have an additional smoothness $p_\G, q_\G\in \dC^1(\G)$, $V_\G\in \dC^0(\G)$,
$p_\g, q_\g\in \dC^1(\g)$, $V_\g\in \dC^0(\g)$ and are analytic
in $\e$ in the norms of these spaces, then we let
$\dC^2(\,\cdot\,):=\bigoplus\limits_{\ed\in\,\cdot\,} C^2(\overline{\ed})\cap W_\infty^2(\ed)$.

Now we are in position to formulate our main result.

\begin{theorem}\label{th1}
Assume that the matrices $\rA_M(\e)$, $\rB_M(\e)$ satisfy the above formulated conditions.
Then the operators $\Op_\e$ and $\Op_0$ are self-adjoint. Suppose that Condition~\ref{C1} holds. The operators $\cR_\G(\e,\l)$ and $\cR_\g(\e,\l)$ are linear and bounded as acting from $L_2(\G)\oplus L_2(\g)$ into $\dC^2(\G)$ and $\dC^2(\g)$. For each $\l\in\mathds{C}\setminus\mathds{R}$ there exists $\e_0(\l)>0$ such that as $\e<\e_0(\l)$, the operators $\cR_\G(\e,\l)$ and $\cR_\g(\e,\l)$ are analytic
in $\e$ as operators from $L_2(\G)\oplus L_2(\g)$ into $\dH^2(\G)$ and $\dH^2(\g)$ and into $\dC^2(\G)$ and $\dC^2(\g)$.
In both cases, the leading terms of the Taylor series for these operators read as
\begin{equation}\label{2.6}
\cR_\G(\e,\l)=(\Op_0-\l)^{-1}\cP_\G +O(\e),
\qquad
\cR_\g(\e,\l)=\cR_\g^{(0)}(\l)\cP_\G+O(\e),
\end{equation}
where the estimates for the error terms depend on $\l$.
\end{theorem}

\section{Discussion and generalizations}

\subsection{Discussion of main results}

In this subsection we discuss the main features of our problem and the main result. First of all we stress that our operator is very general. Namely, its differential expression involves not only the potential, but also the first order terms and a varying coefficient at a higher derivatives, see (\ref{2.1}). On the small edges, the coefficients can depend singularly on $\e$ because of the presence of the negative powers of $\e$ in (\ref{2.1a}). The vertex conditions are also of general form, see (\ref{2.9}). The self-adjointness of the matrix in (\ref{2.4a}) and rank condition (\ref{2.9rank}) are very natural and in fact, they are a criterion ensuring the self-adjointness of the considered operator.

Our main result, Theorem~\ref{th1}, states that under the assumptions made for the coefficients in the differential expression and the vertex conditions and under Condition~\ref{C1}, the resolvent of the operator $\Op_\e$ is analytic
in $\e$ in certain sense. Namely, the operators $\cR_\G(\e,\l)$ and $\cR_\g(\e,\l)$ are analytic
in $\e$. Since the resolvent $(\Op_\e-\l)^{-1}$ can be recovered from the operators $\cR_\G(\e,\l)$ and $\cR_\g(\e,\l)$ via formula (\ref{2.35}), Theorem~\ref{th1} in fact states the analyticity
of the parts of the resolvent $(\Op_\e-\l)^{-1}$ corresponding to subgraphs $\G$ and $\g_\e$. As we have already mentioned, the presence of small edges is a singular perturbation, which can not be treated by the methods from the regular perturbation theory \cite{Kato}. Usually, for singularly perturbed problems one can expect only the convergence result and a chance to construct some asymptotic series. For instance, this is a standard situation for elliptic operators on Euclidean domains or manifolds with singular localized perturbation \cite{Il}, \cite{MNP-book}. A classical example is an elliptic problem in a domain with a small hole, see \cite[Ch. I\!I\!I]{Il}, \cite[Ch. I\!I, Sec. 2.2]{MNP-book}. It is possible to construct complete asymptotic expansions for the corresponding resolvent provided it acts on a sufficiently smooth function, but already simple explicitly solvable examples of such problems show that these asymptotic series are not analytic in the small parameter. From this point of view, our result on analyticity is of a special nature: despite of the presence of a singular perturbation, the parts of the resolvent $(\Op_\e-\l)^{-1}$ are analytic in $\e$ in the operator norms and these norms are strongest possible.

Once we know that the operators $\cR_\G(\e,\l)$ and $\cR_\g(\e,\l)$ are analytic
in $\e$, they are represented by their converging Taylor series. Then we can recover the resolvent $(\Op_\e-\l)^{-1}$ via formula (\ref{2.35}) and obtain a converging Taylor-like series for the resolvent. Namely, given the Taylor series
\begin{equation}\label{3.45}
\begin{aligned}
&\cR_\G(\e,\l)(f_\G,f_\g)=
\sum\limits_{l=0}^{\infty} \e^l u_l^\G, && u_0^\G:=(\Op_0-\l)^{-1} f_\G,
\\
&
\cR_\g(\e,\l)(f_\G,f_\g)=\sum\limits_{l=0}^{\infty} \e^l u_l^\g, \quad && u_0^\g:=\sum\limits_{i=1}^{k} c_i(f_\G) \psi^{(i)},
\end{aligned}
\end{equation}
converging respectively in $\dH^2(\G)\cap \dC^2(\G)$ and $\dH^2(\g)\cap \dC^2(\g)$, we immediately obtain
\begin{equation}\label{2.36}
(\Op_\e-\l)^{-1}f=\sum\limits_{l=0}^{\infty} \e^l u_l^\G\oplus \cS_\e^{-1} u_l^\g,\qquad f_\G:=\cP_\G f,\qquad f_\g:=\cS_\e \cP_{\g_\e} f,
\end{equation}
where the series converges in $\dH^2(\G_\e)\cap \dC^2(\G_\e)$. The remainder in this series can be estimated by means of standard estimates for the remainders of the Taylor series in (\ref{3.45}); one just should take into consideration the following obvious relations:
\begin{equation*}
\big\|f_\G\oplus f_\g \big\|_{L_2(\G)\oplus L_2(\g)}
=\left(\|f\|_{L_2(\G)}^2 +\e^{-1}\|f\|_{L_2(\g_\e)}^2\right)^{\frac{1}{2}}
\leqslant \e^{-\frac{1}{2}}\|f\|_{L_2(\G_\e)}.
\end{equation*}
In particular, in this way we obtain:
\begin{equation}\label{2.50}
\Big\|(\Op_\e-\l)^{-1}f -(\Op_0-\l)^{-1}f\Big\|_{\dH^2(\G)\cap \dC^2(\G)}
\leqslant C\e^{\frac{1}{2}} \|f\|_{L_2(\G_\e)};
\end{equation}
hereinafter the symbol $C$ denotes various inessential constants independent of $f$, $f_\G$, $f_\g$ and $\e$.
By straightforward calculations it is easy to confirm that
\begin{equation*}
\|\cS_\e^{-1} u_0^\g\|_{L_2(\g_\e)}^2=\e \|u_0^\g\|_{L_2(\g)}^2\leqslant C\e \|f_\G\|_{L_2(\G)}^2\leqslant C\e \|f\|_{L_2(\G_\e)}^2,
\end{equation*}
 and then it follows from (\ref{2.36})
that
\begin{equation*}
\big\|(\Op_\e-\l)^{-1}f\big\|_{L_2(\g_\e)}\leqslant C\e^{\frac{1}{2}}\|f\|_{L_2(\G_\e)}.
\end{equation*}
The latter estimate and (\ref{2.50}) are exactly a convergence result similar to ones established in papers \cite{BLS-AdvMath2019} and \cite{Ital}. However, in these papers the difference of the resolvents $(\Op_\e-\l)^{-1}f$ and $(\Op_0-\l)^{-1}f$ was estimated in $L_2$-norm only, while our inequality (\ref{2.50}) provides the bounds in stronger $\dH^2(\G)$-norm and $\dC^2(\G)$-norm. Moreover, we have one more estimate
\begin{equation*}
\bigg\|(\Op_\e-\l)^{-1}f-\cS_\e u_0^\g\bigg\|_{\dC(\g_\e)}
\leqslant C \e^{\frac{1}{2}} \|f\|_{L_2(\G_\e)}.
\end{equation*}
This estimate shows how to approximate the resolvent $(\Op_\e-\l)^{-1}f$ on $\g_\e$ to get a small error term. In particular, in view of formula for $u_0^\g$ in (\ref{3.45}), we see that the resolvent $(\Op_\e-\l)^{-1}f$ is approximated by $\sum\limits_{i=1}^{k} c_i(f_\G) \cS_\e\psi^{(i)}$ in $\dC(\g_\e)$-norm with an error of order $\e^{\frac{1}{2}}$.

In comparison with \cite{BLS-AdvMath2019} and \cite{Ital}, our convergence result has one more important feature. This feature is related with the limiting condition at the vertex $M_0$ for the operator $\Op_0$, namely, with the definition of the matrices $\rA_{M_0}^{(0)}$ and $\rB_{M_0}^{(0)}$. As formulae (\ref{2.32}) show, if the operator $\Op_\infty$ has a virtual level at the bottom of its essential spectrum, the matrix $\rA_{M_0}^{(0)}$ involves the matrix $\rQ$. The presence of the latter produces a Robin part in the vertex condition at the vertex $M_0$. According to formulae (\ref{2.31}), (\ref{2.25a}), (\ref{2.25b}), (\ref{2.25c}), the matrix $\rQ$ is determined by the functions $\frac{d V_\g^{(i)}}{d\e}(\,\cdot\,,0)$ and by the matrices $\rA_M^{(1)}$, $\rB_M^{(1)}$, $M\in\g$.

We also observe that the functions $\frac{d V_\g^{(i)}}{d\e}(\,\cdot\,,0)$ and the matrices $\rA_M^{(1)}$, $\rB_M^{(1)}$ are in fact \textsl{next-to-leading terms} in the above differential expression and vertex conditions. Hence, we conclude that the Robin part in the limiting condition at the vertex $M_0$ is generated \textsl{solely by the above next-to-leading terms}. Once such terms vanish, only the Dirichlet and Neumann parts in the vertex condition at $M_0$ can be present. In particular, this explains why only the Dirichlet and Neumann parts appeared in the limit in \cite{Ital}; the next-to-leading terms were apriori assumed to be zero. The results of \cite{BLS-AdvMath2019} allowed a Robin part in the limiting vertex condition, but in terms of our notations this was because the matrices $\rA_M$ and $\rB_M$ in (\ref{2.60a})
were assumed to be independent of $\e$ and this is why the matrices $\rA_M^{(1)}$ were in general non-zero. However, the issue when and why the Robin part appears in the limiting condition at $M_0$ was not addressed in \cite{BLS-AdvMath2019} and moreover, since the main result in \cite{BLS-AdvMath2019} was formulated in terms of Lagrangian planes, the presence and the origination of the Robin part in the limiting vertex condition was not so explicit as in our case. We also stress that our result shows how the Robin condition can be generated in the general case by the coefficients both in the differential expression and vertex conditions of $\Op_\e$.

In order to demonstrate the difference between our convergence result and those in \cite{BLS-AdvMath2019},~\cite{Ital}, we consider a simple model of a star graph with two fixed edges and one small edge of the length $\e$, which was studied in \cite{Bor3}, see Figure~\ref{example}. On the fixed edges we consider a Schr\"odinger operator $-\frac{d^2}{dx^2}+V_\pm$ with some continuous potentials, while on the small edge the operator reads as $-\frac{d^2}{dx^2}+\e^{-1} V_0(\,\cdot\,\e)$, where $V_0\in C[0,1]$ is some given function. At the vertex $M_0$ the Kirchhoff condition is imposed, while the vertex $M_\e$ is subject to the Dirichlet or Neumann condition. The limiting operator then is just the same operator on finite edges, while at the vertex $M_0$ we have the Dirichlet condition if the same is imposed at $M_\e$. But if the vertex $M_\e$ is subject to the Neumann condition, then in the limit at $M_0$ we have a delta-interaction, that is, the Kirchhoff condition with a coupling:
\begin{equation*}
u\big|_{\ed_-}(M_0)=u\big|_{\ed_-}(M_0)=:u(M_0),
\qquad
u\big|_{\ed_-}(M_0)+u\big|_{\ed_-}(M_0)=c u(M_0),\qquad c:=\int\limits_{0}^{1} V_0(t)\,dt.
\end{equation*}
As we see, the coupling constant $c$ is determined solely by the potential $V_0$.

It is also interesting to compare our Condition~\ref{C1} with similar conditions in~\cite{BLS-AdvMath2019} and~\cite{Ital}. All these conditions are close but there are some differences. In our terms, roughly speaking, Non-resonance condition~3.2 in \cite{BLS-AdvMath2019} says that there are no embedded eigenvalues at the bottom of the essential spectrum for all $\e>0$. Namely,
we should repeat the rescaling we employed while obtaining the operator $\Op_\infty$ in (\ref{2.7}), (\ref{2.18}) but without letting $\e=0$. This gives an operator depending on $\e$ and coinciding with $\Op_\infty$ as $\e=0$. Such operator can not have an embedded eigenvalue at the bottom of its essential spectrum for $\e>0$. This does not imply that the same holds for $\e=0$ and this is the main difference between our Condition~\ref{C1} and Non-resonance condition in \cite{BLS-AdvMath2019}.
Once we make the aforementioned rescaling of small edges for the model studied in \cite{Ital}, we arrive at an operator on the graph $\g_\infty$, which is independent of $\e$ and coincides with our operator $\Op_\infty$.
In this specific situation, in \cite{Ital}, there were considered two cases: generic and non-generic ones. The generic case corresponds to the absence of embedded eigenvalues. In the non-generic case, there can be both virtual level and embedded eigenvalues at the bottom of the essential spectrum. But since the model in \cite{Ital} was quite specific, it was possible to study the the convergence of the resolvent in the non-generic case as well.

In the present work we do not study the case when Condition~\ref{C1} is violated and the operator has an embedded eigenvalue at the bottom of the essential spectrum. The reason is that although our technique is sufficient to treat such case, it involves more technical details. The nature of the expected results also suggests that this case deserves an independent study, which we postpone for the next work.

\subsection{Generalizations of model}

Here we discuss
possible ways of extending our results to more general models.

The first
possibility concerns a case, when instead of one graph $\g_\e$, several similar small rescaled graphs $\g_\e^{(j)}$ are glued at
different vertices of the graph $\G$, see Figure~\ref{fig3}. Each of the small graphs $\g_\e^{(j)}$ can be introduced via rescaling by means of an associated small parameter $\e^{(i)}$; then coefficients $p_\e^{(j)}$, $q_\e^{(j)}$, $V_\e^{(j)}$ of the differential expression on each such graph $\g_\e^{(j)}$ are still defined by formula (\ref{2.1a}), but now
on each graph $\g_\e^{(j)}$ the rescaling is made by means of a corresponding operator $\cS_\e^{(j)}$. We can suppose first that all parameters $\e^{(j)}$ are analytic
functions in a single small parameter $\e$. In this situation, all our results remain valid with obvious minor changes; in particular, the form of the limiting vertex condition at a vertex $M_0^{(j)}$, to which the graph $\g_\e^{(j)}$ is attached, is determined only by the graph $\g_\e^{(j)}$ and there is no influence by the neighbouring ones.

\begin{figure}
\begin{center}
\includegraphics[scale=0.7]{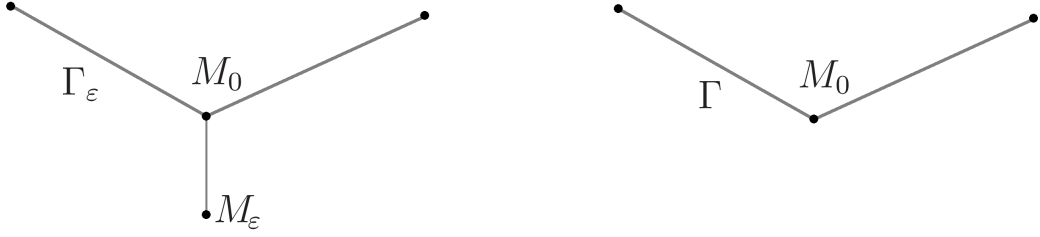}
\end{center}
\caption{{\small Star graph with a small edge and the limiting graph without the small edge.}}
\label{example}
\end{figure}

A more general situation is when in the same setting we assume that the parameters $\e^{(j)}$ are independent and all coefficients in the differential expression and vertex conditions of the operator $\Op_\e$ depend analytically in all $\e^{(j)}$. In such case, we first of all should assume additionally that for each attached graph $\g_\e^{(j)}$ and each vertex $M\in\g_\infty^{(j)}$, identities similar to (\ref{2.60a}) hold:
\begin{align*}
&\e^{(j)}\tilde{\rE}_M(\e^{(j)}) \rA_M(\e)=\big(\e^{(j)}\big)^\mu\big(\rA_M^{(0)}+\e^{(j)} \rA_M^{(1)}\big)+O(|\e|^{2+\mu}),
\\
&
\tilde{\rE}_M(\e^{(j)}) \rB_M(\e)=\big(\e^{(j)}\big)^\mu\big(\rB_M^{(0)}+\e^{(j)} \rB_M^{(1)}\big)+O(|\e|^{2+\mu}),
\end{align*}
where $\mu=0$
if $\rB_M(0)\ne0$ and $\mu=1$
if $\rB_M(0)=0$ and $\e=(\e^{(1)}, \e^{(2)},\ldots)$. For the coefficients of the differential expression we should assume that
\begin{equation}\label{2.64}
p_{\g^{(j)}}(\,\cdot\,,\e)=p_{\g^{(j)}}(\,\cdot\,,0) + \e^{(j)} \frac{d p_{\g^{(j)}}}{d\e^{(j)}}(\,\cdot\,,0) + O(|\e|^2),
\end{equation}
and the same should hold for the other coefficients.
Then the corresponding matrices $Q^{(j)}$ are to be calculated for each graph $\g^{(j)}$ by formulae (\ref{2.31}), (\ref{2.25a}), (\ref{2.25b}), (\ref{2.25c}). If for some $j$ the matrix $Q^{(j)}$ turns out to have a zero eigenvalue, then assumptions (\ref{2.60a}), (\ref{2.64})
should be replaced by stricter ones:
\begin{align*}
&\e^{(j)}\tilde{\rE}_M(\e^{(j)}) \rA_M(\e)=\big(\e^{(j)}\big)^\mu\big(\rA_M^{(0)}+\e^{(j)} \rA_M^{(1)}\big)+O(|\e|^{2+\mu}),
\\
&\tilde{\rE}_M(\e^{(j)}) \rB_M(\e)=\big(\e^{(j)}\big)^\mu\big(\rB_M^{(0)}+\e^{(j)} \rB_M^{(1)}\big)+O(|\e|^{2+\mu}),
\\
&p_{\g^{(j)}}(\,\cdot\,,\e)= p_{\g^{(j)}}(\,\cdot\,,0) + \e^{(j)} \frac{d p_{\g^{(j)}}}{d\e^{(j)}}(\,\cdot\,,0)+ \frac{\big(\e^{(j)}\big)^2}{2} \frac{d^2 p_{\g^{(j)}}}{d(\e^{(j)})^2}(\,\cdot\,,0) + O(|\e|^3).
\end{align*}
The above conditions mean that the coefficients in the differential expression and vertex conditions can depend on all small parameters, but the leading terms in their Taylor should be as specified above. Under such assumptions, all our results remain true and now the analyticity holds with respect to all parameters $\e^{(i)}$. Once these assumptions fail, in the general situation, the result on analyticity
 is no longer valid. The reason for this situation is in fact a well-known phenomenon in the multi-parametric perturbation theory, which says that the eigenvalues of a matrix analytic with respect to more than one small parameter are not necessary analytic, see
\cite[Ch. I\!I, Sect. 5.7]{Kato}. The same also concerns the case, when several graphs $\g^{(j)}$ are rescaled by means of different independent small parameters and are glued to the same vertex in the graph $\G$. In such general situation one can expect only a convergence result as one proved in \cite{BLS-AdvMath2019}. Of course, this does not exclude the situations, when for some particular graphs and specific choices of small edges, differential expression and vertex conditions, the resolvent has the analyticity property, but this can happen only due to some specific features of the considered models.

\begin{figure}
\begin{center}
\includegraphics[scale=0.7]{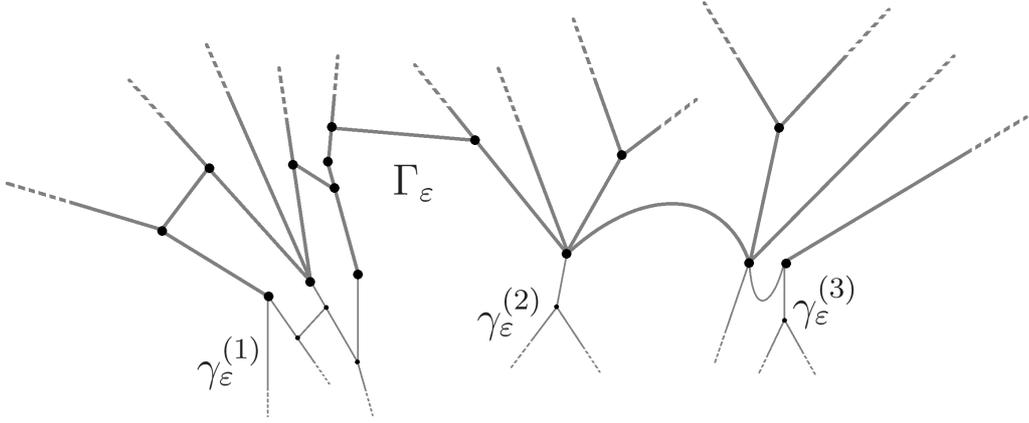}
\end{center}
\caption{{\small Graph $\G$ with several glued small graphs}}\label{fig3}
\end{figure}

One more important way of extending our model is to assume that all coefficients in the differential expression and the vertex conditions are not analytic in $\e$ but either infinite differentiable or just possess power asymptotic expansions in $\e$. In these cases all our results, including ones just discussed above for several glued small graphs, remain true with appropriate modifications. Namely, if the mentioned coefficients are infinite differentiable in $\e$, then the statement on the analyticity property is to be replaced by the infinite differentiability in $\e$ for the operators $\cR_\G$ and $\cR_\g$. If the coefficients in the differential expression and the vertex conditions just possess power asymptotic expansions in $\e$, then the operators $\cR_\G$ and $\cR_\g$ also have power asymptotic expansions in $\e$. The way of proving such results is very simple and is based on the analyticity case treated above. Namely, we represent the matrices in the vertex conditions as
\begin{equation}\label{2.68}
\rA_M(\e)=\rA_M^{(0)}+\e \rA_M^{(1)} + \e^2 \rA_M^{(2)} + \e^3 \breve{\rA}_M(\e),
\qquad
\rB_M(\e)=\rB_M^{(0)}+\e \rB_M^{(1)} + \e^2 \rB_M^{(2)} + \e^3 \breve{\rB}_M(\e),
\end{equation}
where the matrices $\breve{\rA}_M(\e)$ and $\breve{\rB}_M(\e)$ are either infinitely differentiable in $\e$ or have power in $\e$ asymptotic expansions. For the coefficients of the differentiable expression, a similar representation is to be employed. The terms $\breve{\rA}_M$ and $\breve{\rB}_M$ are to be treated as frozen coefficients independent of $\e$ and the matrices in (\ref{2.68}) should be treated as second order polynomials in $\e$. Then it is possible to track the dependence on $\breve{\rA}_M$ and $\breve{\rB}_M(\e)$ and on similar terms for the coefficients in the differential expression in the proof of Theorem~\ref{th1}. As a result, then one can see that the operators $\cR_\G$ and $\cR_\g$ as well as the resolvent $(\Op_\e-\l)^{-1}$ are represented by convergent series (\ref{3.45}) and (\ref{2.36}), but now with coefficients depending on $\e$. These coefficients are either infinitely differentiable in $\e$ or possess power in $\e$ asymptotic expansions. In the first case we conclude immediately that all the operators in question are infinitely differentiable in $\e$, while in the other case we can easily find the asymptotic expansions for these operators from their series.

\section{Self-adjointness}

Our main result includes a statement on the self-adjointness of both perturbed and limiting operators and in the proofs, we shall also use the same property for certain auxiliary operators.
In this section we first establish a criterion of the self-adjointness for a general operator. Then we shall employ it to prove the same property for the operators $\Op_0$ and $\Op_\g$.

\subsection{Criterion of self-adjointness}

Let $\Xi$ be a finite metric graph having no isolated vertices; the lengths of its edges can be both finite or infinite. On each edge we fix arbitrarily an orientation and as above, by $d(M)$ we denote the degree of a vertex $M\in\Xi$.

We consider an unbounded operator $\Op$ in $L_2(\Xi)$ with the differential expression
\begin{equation*}
\hat{\Op}u:=-\frac{d\ }{dx} p\frac{d u}{dx} + \iu \left(
\frac{d\ }{dx} (q u) + q\frac{d u}{dx}
\right)+ V u
\end{equation*}
subject to the vertex conditions
\begin{equation}\label{7.2}
\rA_M\cU_M(u) +\rB_M \cU_M'(u)=0\quad\text{at each vertex $M$}.
\end{equation}
Here $p,\,q\in\Hinf^1(\Xi)$, $V\in L_2(\Xi)$ are real-valued functions, while $\rA_M$ and $\rB_M$ are some matrices of size $d(M)\times d(M)$. The ellipticity condition is assumed:
$p\geqslant c>0$ uniformly on $\Xi$.
The domain of the operator $\Op$ consists of the functions in $\dH^2(\Xi)$ obeying vertex condition (\ref{7.2}).

By $\Pi_M$, $\Th_M$ we denote the following diagonal matrices with real entries:
\begin{equation*}
\Pi_M:=\diag\big\{\nu_i(M) p\big|_{\ed_i(M)}(M,\e)\big\}_{i=1,\ldots,d(M)}, \qquad
\Th_M:=\diag\big\{\nu_i(M) q\big|_{\ed_i(M)}(M,\e)\big\}_{i=1,\ldots,d(M)}, \end{equation*}
where $\ed_i(M)$ are the edges incident to the vertex $M$ and the numbers $\nu_i(M)$ are defined in the same way as in (\ref{2.4b}).

The main statement of this subsection is the following lemma.

\begin{lemma}\label{lm7.1}
The operator $\Op$ is self-adjoint if and only if for each vertex $M\in\Xi$ the matrices $\rA_M$ and $\rB_M$ satisfy the rank condition
\begin{equation}\label{7.3}
\rank (\rA_M\ \rB_M)=d(M)
\end{equation}
and the matrix
\begin{equation}\label{7.4}
\rA_M\Pi_M^{-1} \rB_M^* +\iu \rB_M\Pi_M^{-1}\Th_M \Pi_M^{-1}\rB_M^*
\end{equation}
is self-adjoint. Under these restrictions, vertex condition (\ref{7.2}) can be equivalently rewritten as
\begin{equation}\label{7.6}
\iu (\rU_M-\rI_{d(M)})\cU_M(u)+(\rU_M+\rI_{d(M)}) \big(\Pi_M\cU_M'(u)
-\iu \Th_M\cU_M(u)\big)=0,
\end{equation}
where the matrix
\begin{equation*}
\rU_M:=-\big(\rA_M+\iu \rB_M\Pi_M^{-1}(\Th_M-\rI_{d(M)})
\big)^{-1}\big(\rA_M+\iu \rB_M\Pi_M^{-1}(\Th_M+\rI_{d(M)})\big)
\end{equation*}
is well-defined and is unitary.
\end{lemma}
\begin{proof}
In a particular case $p\equiv 1$, $q\equiv0$, $V\equiv0$, the lemma states a well-known result, see Theorem~1.1.4 in \cite[Ch. 1, Sect. 1.4.]{BK-book}. In fact, the proof of this theorem can be adapted also for our more general case up to some minor changes. Below we describe how to do this.

First of all we observe that vertex condition (\ref{7.2}) can be equivalently rewritten as
\begin{gather}\label{7.9a}
\tilde{\rA}_M \cU_M(u)+\tilde{\rB}_M
\big(\Pi_M\cU_M'(u) -\iu\Th_M\cU_M(u)\big)=0,
\\
\tilde{\rA}_M:=\rA_M+\iu \rB_M \Pi_M^{-1} \Th_M,\qquad \tilde{\rB}_M:=\rB_M \Pi_M^{-1}.
\label{7.9b}
\end{gather}
Let $u$ be an arbitrary function from the domain of the operator $\Op$ and $v$ be an arbitrary function from $\bigoplus\limits_{\ed\in\Xi} C^\infty(\ed)$ vanishing outside a small neighbourhood of an arbitrarily fixed vertex $M\in\Xi$. Integrating twice by parts, we get:
\begin{align*}
(\Op u,v)_{L_2(\Xi)}=&\big(\Pi_M\cU_M'(u),\cU_M(v)\big)_{\mathds{C}^{d(M)}} - \big(\cU_M(u),\Pi_M\cU_M'(v)\big)_{\mathds{C}^{d(M)}}+(u,\hat{\Op} v)_{L_2(\Xi)}
\\
=&(\mathsf{F}_M',\mathsf{G}_M)_{\mathds{C}^{d(M)}}-(\mathsf{F}_M,\mathsf{G}_M')_{\mathds{C}^{d(M)}}
+ (u,\hat{\Op} v)_{L_2(\Xi)},
\end{align*}
where
\begin{align*}
&\mathsf{F}_M':=\Pi_M\cU_M'(u)-\iu\Th_M\cU_M(u), \qquad \mathsf{G}_M:=\cU_M(v),
\\
&\mathsf{F}_M:=\cU_M(u),\qquad \mathsf{G}_M':=\Pi_M\cU_M'(v)-\iu\Th_M\cU_M(v).
\end{align*}
Hence, the function $v$ is in the domain of the adjoint operator $\Op^*$ if
\begin{equation*}
(\mathsf{F}_M',\mathsf{G}_M)_{\mathds{C}^{d(M)}}-(\mathsf{F}_M,\mathsf{G}_M')_{\mathds{C}^{d(M)}}=0.
\end{equation*}
The latter identity is exactly equation (1.4.16) in the proof of Theorem~1.1.4 in \cite[Ch. 1, Sect. 1.4]{BK-book} but with different $\mathsf{F}_M$, $\mathsf{F}_M'$, $\mathsf{G}_M$, $\mathsf{G}_M'$. However, the specific form of these quantities were not important in that proof in \cite{BK-book}. This is why the arguing from \cite{BK-book} can be reproduced literally with above introduced $\mathsf{F}_M$, $\mathsf{F}_M'$, $\mathsf{G}_M$, $\mathsf{G}_M'$; as the matrices $A$ and $B$ used in that proof, our matrices $\tilde{\rA}_M$ and $\tilde{\rB}_M$ serve. Rank condition
(\ref{7.3}) is equivalent to similar condition (1.4.10) in \cite[Ch. 1, Sect. 1.4]{BK-book} since thanks to the non-degeneracy of the matrix $\Pi_M^{-1}$ we have
\begin{align*}
\rank (\tilde{\rA}_M\ \;\tilde{\rB}_M)=&\rank\big(
\rA_M+\iu \rB_M \Pi_M^{-1} \Th_M\ \;\rB_M \Pi_M^{-1}\big)
\\
=&\rank\big(
\rA_M \ \;\rB_M \Pi_M^{-1}\big)=\rank\big(
\rA_M \ \;\rB_M\big)=d(M).
\end{align*}
Condition (1.4.11) for the matrices $A$ and $B$ from Theorem~1.1.4 in \cite[Ch. 1, Sect. 1.4.]{BK-book} is exactly our condition (\ref{7.4}), while equivalent formulation
 (\ref{7.6})
 of the vertex condition are implied immediately by similar formulations (1.4.13)
 in \cite{BK-book}. The proof is complete.
\end{proof}

\subsection{Self-adjointness of operators $\Op_0$ and $\Op_\infty$}\label{sec6.1}

In this subsection we prove that under the assumptions of
Theorem~\ref{th1} the operators $\Op_0$ and $\Op_\infty$ are self-adjoint. In order to prove this statement for the operator $\Op_0$,
we shall need the self-adjointness of the matrix $\rQ$ involved in the matrix $\rA_{M_0}^{(0)}$, see (\ref{2.32}). This fact is provided by the following lemma, which will be independently proved later in a separate subsection, see Subsection~\ref{secQ}.
\begin{lemma}\label{lmQ}
The matrix $\rQ$ is self-adjoint.
\end{lemma}

Let us show that the operators $\Op_0$ and $\Op_\infty$ are self-adjoint.
The matrix in (\ref{2.4a}) is self-adjoint as $\e=0$ for each vertex $M\in\G$, $M\ne M_0$. Hence, the matrices $\rA_M^{(0)}$, $\rB_M^{(0)}$ satisfy condition (\ref{7.3}), (\ref{7.4}) for each vertex $M\in\G$, $M\ne M_0$. Thanks to the self-adjointness of the matrix $\rQ$ stated in Lemma~\ref{lmQ}, the self-adjointness of the matrix in (\ref{7.4}) corresponding to the matrices $\rA_{M_0}^{(0)}$ and $\rB_{M_0}^{(0)}$ is confirmed by simple straightforward calculations. In view of the non-degeneracy of the matrix $\Pi_{\G,M_0}(0)$, condition (\ref{7.3}) corresponding to the same matrices is checked as follows:
\begin{align*}
\rank\big(\rA_{M_0}^{(0)}\ \rB_{M_0}^{(0)}\big)=&\rank \Big(\rA_{M_0}^{(0)}+\iu\rB_{M_0}^{(0)} \Pi_{\G,M_0}^{-1}(0)\Th_{\G,M_0}(0)\Psi\quad -\rB_{M_0}^{(0)}\Pi_{\G,M_0}^{-1}(0)\Psi\Big)
\\
=&\rank \left(\begin{pmatrix}
\rQ & 0
\\
0 & \rI_{d_0-k}
\end{pmatrix}\quad \begin{pmatrix}
\rI_k & 0
\\
0 & 0
\end{pmatrix}\right)=d_0.
\end{align*}
Hence, we can apply Lemma~\ref{lm7.1} and we see that the operator $\Op_0$ is self-adjoint.

We proceed to proving the self-adjointness of the operator $\Op_\infty$. In view of the assumption on the matrices $\rA_M(0)$ and $\rB_M(0)$ made in Subsection~\ref{sec2.1}, see the explanation after (\ref{2.9}), we see easily that rank condition (\ref{2.9rank}) implies the same for the matrices $\rA_M^{(0)}$ and $\rB_M^{(0)}$ defined in (\ref{2.11b}).
Then we multiply the matrix in (\ref{2.4a}) by $\e$ and then by $\tilde{\rE}_M(\e)$ from left and right keeping its self-adjointness and then we pass to the limit as $\e\to+0$. By identities (\ref{2.60a}) we then arrive to the self-adjointness of the matrix
\begin{equation*}
\rA_M^{(0)}\Pi_{\g,M}^{-1}(0) \big(\rB_M^{(0)}\big)^*+\iu \rB_M^{(0)}\Pi_{\g,M}^{-1}(0)\Th_{\g,M}(0) \Pi_{\g,M}^{-1}(0)\big(\rB_M^{(0)}\big)^*.
\end{equation*}
Applying now Lemma~\ref{lm7.1}, we conclude that the operator $\Op_\infty$ is self-adjoint.

\subsection{Self-adjointness of $\rQ$}\label{secQ}

In this section we prove Lemma~\ref{lmQ}. The proof is based on obtaining certain representations for the operators $\cL_M$ and $\cQ_M$ used in (\ref{2.25b}), (\ref{2.25c}), (\ref{2.29}), which will imply then the desired self-adjointness. These representations are based on an appropriate transformation of vertex conditions (\ref{2.9}) at $M\in\g_\e$ and we begin our proof with describing them.

Given a vertex $M\in\g_\e$, including vertices $M=M_j$, $j=1,\ldots,n$, we rescale the variables as $\xi=x\e^{-1}$ passing in fact to the graph $\g_\infty$. Then vertex condition (\ref{2.9}) becomes
\begin{equation}\label{5.2}
\e\rA_M(\e)\cU_M(u)+\rB_M(\e)\cU_M'(u)=0,
\end{equation}
where the vectors $\cU_M(u)$ and $\cU_M'(u)$ are introduced as above but are treated in terms of the variable $\xi$. According Lemma~\ref{lm7.1}, at each vertex $M\in\g_\infty$ the above vertex condition can be equivalently rewritten as
\begin{equation}\label{4.2}
\iu (\rU_M(\e)-\rI_{d(M)})\cU_M(u)+(\rU_M(\e)+\rI_{d(M)}) \big(\Pi_{\g,M}(\e)\cU_M'(u)
-\iu \Th_{\g,M}(\e)\cU_M(u)\big)=0,
\end{equation}
where, we recall, the matrices $\Pi_{\g,M}$ and $\Th_{\g,M}$ were defined in (\ref{2.26a}), and
\begin{equation}
\label{2.5a}
\begin{aligned}
\rU_M(\e):=-&\Big(\e\rA_M(\e)+\iu \rB_M(\e)\Pi_{\g,M}^{-1}(\e)
\big( \Th_{\g,M}M(\e)-\rI_{d(M)}\big)\Big)^{-1}
\\
&\cdot\Big(\e\rA_M(\e)+\iu \rB_M(\e)\Pi_{\g,M}^{-1}(\e)
\big( \Th_{\g,M}(\e)+\rI_{d(M)}\big)\Big)
\end{aligned}
\end{equation}
is an unitary matrix.

\begin{lemma}\label{lm4.2a}
 For all $M\in\G$, $M\ne M_0$ and all $M\in\g_\infty$, the matrix $\rU_M(\e)$ is analytic in $\e$.
\end{lemma}
\begin{proof} In view of
rank condition (\ref{2.9rank}) and formulae (\ref{2.60a}) and owing to the analyticity of
$\rA_M(\e)$ and $\rB_M(\e)$
in $\e$, the matrix $\rU_M(\e)$ is meromorphic in $\e$. At the same time, by the unitarity of $\rU_M(\e)$
we have $\|\rU_M(\e)\rx\|_{\mathds{C}^{d(M)}}= \|\rx\|_{\mathds{C}^{d(M)}}$ for all $\rx\in \mathds{C}^{d(M)}$ and sufficiently small $\e$. Hence, the matrix $\rU_M(\e)$ is bounded uniformly in $\e$ and can not have poles. Therefore, it is analytic in $\e$. The proof is complete.
\end{proof}

Thanks to the unitarity of the matrix $\rU_M(\e)$, by the results in \cite[Ch. I\!I, Sec. 4.6]{Kato}, the eigenvalues and the associated orthonormalized in $\mathds{C}^{d(M)}$ vectors of the matrix $\rU_M(\e)$ are analytic in $\e$. By $\rP_M(\e)$ we denote the total projector in $\mathds{C}^{d(M)}$ onto the eigenspace associated with the eigenvalues of the matrix $\rU_M(\e)$ converging to $-1$ as $\e\to+0$. We also let $\rP_{M,\bot}(\e):=\rI_{d(M)}-\rP_M(\e)$. By the aforementioned analyticity in $\e$ of the eigenvalues and the eigenvectors of $\rU_M(\e)$,
the projectors $\rP_M(\e)$ and $\rP_{M,\bot}(\e)$ are also analytic in $\e$.

We apply the projectors $\rP_M(\e)$ and $\rP_{M,\bot}(\e)$ to vertex condition (\ref{4.2}) and we rewrite it equivalently as
\begin{equation}\label{4.1}
\begin{aligned}
&
\rP_M(\e)\cU_M(u) + \rK_M(\e) \big(\Pi_{\g,M}(\e) \cU_M'(u)-\iu \Th_{\g,M}(\e)\cU_M(u)\big)=0,
\\
&\rP_{M,\bot}(\e) \Pi_{\g,M}(\e)\cU_M'(u)+\big(\rK_{M,\bot}(\e) - \iu \rP_{M,\bot}(\e)\Th_{\g,M}(\e) \big)\cU_M(u)=0,
\end{aligned}
\end{equation}
where
\begin{equation}
\begin{aligned}
&\rK_M(\e):=-\iu \big(\rU_M(\e)-\rI_{d(M)}\big)^{-1}\rP_M(\e) \big(\rU_M(\e)+\rI_{d(M)}\big)
\\
&\rK_{M,\bot}(\e):=\iu \big(\rU_M(\e)+\rI_{d(M)}\big)^{-1}\rP_{M,\bot}(\e) \big(\rU_M(\e)-\rI_{d(M)}\big).
\end{aligned}
\label{3.4}
\end{equation}
It is obvious that both matrices $\rK_M(\e)$ and $\rK_{M,\bot}(\e)$ are well-defined and analytic in $\e$. Employing the unitarity of the matrix $\rU_M(\e)$, it is straightforward to confirm that both matrices $\rK_M(\e)$ and $\rK_{M,\bot}(\e)$ are self-adjoint. We also observe that
\begin{equation*}
\rU_M(0)=\rU_M^{(0)},\qquad \rP_M(0)=\rP_M^{(0)}, \qquad \rK_M(0)=0.
\end{equation*}

Since the operator $\rP_M(\e)$ is analytic in $\e$, by the results in \cite[Ch. 2, Sect. 4.2]{Kato}, there exists a transforming function for this projector, namely, an invertible operator $\rS_M(\e)$ in $\mathds{C}^{d(M)}$, analytic in small $\e$ together with its inverse operator such that
\begin{equation}\label{4.13}
\rS_M(0)=\rI_{d(M)},\qquad \rS_M^{-1}(\e)\rP_M^{(0)}\rS_M(\e)=\rP_M(\e),\qquad
\rS_M^{-1}(\e)\rP_{M,\bot}^{(0)}\rS_M(\e)=\rP_{M,\bot}(\e)
\end{equation}
The inverse operator $\rS_M^{-1}(\e)$ is given explicitly by formula (4.18) in \cite[Ch. 2, Sect. 4.2]{Kato} and this implies
\begin{equation}\label{3.8a}
\rS_M(\e)=\big(\rP_M(\e)\rP_M^{(0)} +\rP_{M,\bot}(\e)\rP_{M,\bot}^{(0)}\big)^{-1} \Big(\rI_{d(M)}
-\big(\rP_M(\e)-\rP_M^{(0)}\big)^2\Big)^{\frac{1}{2}}.
\end{equation}

In view of formulae (\ref{4.13}), (\ref{3.8a}), vertex condition (\ref{4.1}) can be equivalently rewritten as
\begin{gather*}
\rP_M^{(0)}\rS_M(\e)\cU_M(u)+\rS_M(\e) \rK_M(\e) \big(\Pi_{\g,M}(\e) \rU_M'(\e)-\iu \Th_{\g,M}(\e)\cU_M(u)\big)=0,
\\
\rP_{M,\bot}^{(0)}\rS_M(\e)\Pi_{\g,M}M(\e) \cU_M'(u)+\rS_M(\e)\big(\rK_{M,\bot}(\e) -\iu \rP_{M,\bot}(\e)\Th_{\g,M}(\e)
\big)\cU_M(u)=0.
\end{gather*}
Since the first terms in the above identities belong respectively to the spaces $\rP_M^{(0)}\mathds{C}^{d(M)}$ and $\rP_{M,\bot}^{(0)}\mathds{C}^{d(M)}$, the same is true for the other terms in these identities. Therefore, we can apply the projector $\rP_{M,\bot}^{(0)}$ to these identities and we equivalently rewrite them as
\begin{equation}\label{4.16}
\rP_M^{(0)}\tilde{\rS}_M(\e)\cU_M(u) +\tilde{\rK}_M(\e)\cU_M'(u)=0,\qquad
\rP_{M,\bot}^{(0)}\tilde{\rS}_{M,\bot}(\e)\cU_M'(u)+ \tilde{\rK}_{M,\bot}(\e)\cU_M(u)=0,
\end{equation}
where
\begin{equation}\label{4.17}
\begin{aligned}
&\tilde{\rS}_M(\e):=\rS_M(\e)\big(\rI_{d(M)}-\iu \rK_M(\e) \Th_{\g,M}(\e)\big),
\qquad \tilde{\rK}_M(\e):=\rP_M^{(0)} \rS_M(\e) \rK_M(\e) \Pi_{\g,M}(\e),
\\
&\tilde{\rS}_{M,\bot}(\e):=\rS_M(\e)\Pi_{\g,M}(\e),\qquad
\tilde{\rK}_{M,\bot}(\e):=\rP_{M,\bot}^{(0)} \rS_M(\e)\big(\rK_{M,\bot}(\e) -\iu \rP_{M,\bot}(\e)\Th_{\g,M}(\e)
\big).
\end{aligned}
\end{equation}

The matrices $\rS_M(\e)$, $\tilde{\rS}_M(\e)$, $\tilde{\rS}_{M,\bot}(\e)$, $\tilde{\rK}_M(\e)$, $\tilde{\rK}_{M,\bot}(\e)$, $\Pi_{\g,M}(\e)$, $\Th_{\g,M}(\e)$ are analytic in $\e$ and by $\rS_{M,1}$, $\tilde{\rS}_{M,1}$, $\tilde{\rS}_{M,\bot,1}$,
$\tilde{\rK}_{M,1}$, $\tilde{\rK}_{M,\bot,1}$, $\Pi_{\g,M,1}(\e)$, $\Th_{\g,M,1}(\e)$ we denote their derivatives in $\e$ taken at $\e=0$.
For instance,
\begin{equation*}
\tilde{S}_{M,1}:= \frac{d \tilde{\rS}_M}{d\e }(0),\qquad
\tilde{S}_{M,\bot,1}:= \frac{d \tilde{\rS}_{M,\bot}}{d\e}(0).
\end{equation*}

In order to prove Lemma~\ref{lmQ}, we shall need several properties of the introduced matrices, which are summarized in the following auxiliary statement.

\begin{lemma}\label{lm7.2}
For each vertex $M\in\g_\infty$, the identities hold:
\begin{align}\label{5.8}
&
\begin{aligned}
&\rP_{M,\bot}^{(0)}\rP_{M,1}=\rP_{M,1}\rP_M^{(0)},\qquad \rP_{M,\bot}^{(0)}\rP_{M,1}\rP_{M,\bot}^{(0)}=0,
\\
&\rP_M^{(0)}\rP_{M,1}=\rP_{M,1}\rP_{M,\bot}^{(0)},\qquad \rP_M^{(0)}\rP_{M,1}\rP_M^{(0)}=0,\qquad
\rP_{M,1}^*=\rP_{M,1},
\end{aligned}
\\
&
\begin{aligned} \label{4.63}
\rS_{M,1}=\rP_{M,1}\rP_{M,\bot}^{(0)}-\rP_{M,1}\rP_M^{(0)} =\rP_{M,1}\rP_{M,\bot}^{(0)}- \rP_{M,\bot}^{(0)}\rP_{M,1},
\end{aligned}
\\
&\rP_M^{(0)}\rK_{M,1}\rP_{M,1}^{(0)}=0.\label{5.18a}
\end{align}
\end{lemma}

\begin{proof}
Formulae (\ref{5.8}) are obtained immediately via calculating the coefficients at $\e^1$ in the obvious identities:
\begin{equation*}
\big(\rP_{M,\bot}(\e)\big)^2=\rP_{M,\bot}(\e), \qquad \big(\rP_M(\e)\big)^2=\rP_M(\e), \qquad
\big(\rP_{M,\bot}(\e)\big)^*=\rP_{M,\bot}(\e).
\end{equation*}
Formula (\ref{4.63}) can be confirmed by expanding the right hand side in (\ref{3.8a}) into the power series in $\e$, calculating the coefficient at $\e^1$, and employing then (\ref{5.8}). It follows from the formula for $\rK_M(\e)$ in (\ref{3.4}) that
\begin{align*}
\iu(\rU_M(\e)-\rI_{d(M)}) \rP_M(\e) \rK_M(\e) \rP_{M,\bot}(\e)=\rP_M(\e) (\rU_M(\e)+\rI_{d(M)})\rP_{M,\bot}(\e)=0.
\end{align*}
Expanding this identity into the power series in $\e$ and calculating the coefficient at $\e^1$, we arrive at (\ref{5.18a}). The proof is complete.
\end{proof}

Now we are in position to prove Lemma~\ref{lmQ}.

\begin{proof}[Proof of Lemma~\ref{lmQ}]

According definition (\ref{2.5a}) of $\rU_M(\e)$, formulae (\ref{7.9a}), (\ref{7.9b}) applied with \begin{equation*}
\tilde{\rA}_M=\tilde{\rA}_M(\e):=\e\rA_M(\e)+\iu \rB_M(\e) \Pi_{\g,M}^{-1}(\e) \Th_{\g,M}(\e),\qquad \tilde{\rB}_M=\tilde{\rB}_M(\e):=\rB_M(\e) \Pi_{\g,M}^{-1}(\e),
\end{equation*}
and the easily checked identities
\begin{align*}
&-2\iu\big(\tilde{\rA}_M(\e)-\iu \tilde{\rA}_M(\e)\big)^{-1}\tilde{\rA}_M(\e)=\iu\big(\rU_M(\e)-\rI_{d(M)}\big),
\\
&-2\iu\big(\tilde{\rA}_M(\e)-\iu \tilde{\rA}_M(\e)\big)^{-1}\tilde{\rB}_M(\e)=\rU_M(\e)+\rI_{d(M)},
\end{align*}
for each $\psi^{(i)}$ and each vertex $M\in\g_\infty$ we have:
\begin{equation}\label{5.12}
\begin{aligned}
\e\rA_M(\e)&\cU_M(\psi^{(i)})+\rB_M(\e)\cU_M'(\psi^{(i)})= \frac{\iu}{2} \big(\tilde{\rA}_M(\e)-\iu\tilde{\rB}_M(\e)\big) \bigg( \iu (\rU_M(\e)-\rI_{d(M)})\rP_M(\e)
\\
&\cdot\Big(
\big(\rI_{d(M)}-\iu\rK_M(\e)\Th_{\g,M}(\e) \big) \cU_M(\psi^{(i)})
 +
\rK_M(\e)\Pi_{\g,M}(\e) \cU_M'(\psi^{(i)})
\Big)
\\
& +
\big(\rU_M(\e)+\rI_{d(M)}\big)\rP_{M,\bot}(\e) \big(
 \Pi_{\g,M}(\e)\cU_M'(\psi^{(i)})+ \big(\rK_{M,\bot}(\e)-\iu \Th_{\g,M}(\e)\big)\cU_M(\psi^{(i)})\big)
\bigg).
\end{aligned}
\end{equation}
The latter long term in the right hand side in the above identity becomes the assumed vertex condition for $\psi^{(i)}$ as $\e=0$ and hence, it vanishes as $\e=0$. We expand then identity (\ref{5.12}) into the Taylor series in $\e$ and calculate the coefficient at $\e^1$ if $\rB_M(0)\ne0$ and at $\e^2$ if $\rB_M(0)=0$. Employing then the vertex condition for $\psi^{(i)}$, we obtain:
\begin{align*}
\cQ_M(\psi^{(i)})=&\iu (\rU_M^{(0)}-\rI_{d(M)}) \rP_M^{(0)}\Big(
\rP_M^{(0)} \rK_{M,1} \cV_M(\psi^{(i)}) + \rP_{M,1} \cU_M(\psi^{(i)})
\Big)
\\
&
+ (\rU_M^{(0)}+\rI_{d(M)})\rP_{M,\bot}^{(0)}\rP_{M,1}
\Big(\cV_M(\psi^{(i)})+
\rK_{M,\bot}(0)
\cU_M(\psi^{(i)})
\\
&+\Pi_{\g,M,1}\cU_M'(\psi^{(i)})
+ \big(\rK_{M,\bot,1}-\iu \Th_{\g,M,1}
\big)\rP_{M,\bot}^{(0)}\cU_M(\psi^{(i)})
\Big).
\end{align*}
Applying the matrices $\rP_M^{(0)}$ and $(\rU_M^{(0)}+\rI_{d(M)})^{-1}\rP_{M,\bot}^{(0)}$ to the above identity and using formulae (\ref{5.8}) and $\rP_M^{(0)}\rK_{M,\bot}(0)=0$, $\rP_M^{(0)}\cU_M(\psi^{(i)})=0$,
we arrive at
\begin{align}
&\label{5.9}
\begin{aligned}
\rP_M^{(0)}\cQ_M(\psi^{(i)})=-2\iu\rP_M^{(0)}\rK_{M,1} \cV_M(\psi^{(i)})
-2\iu \rP_M^{(0)}
\rP_{M,1}\rP_{M,\bot}^{(0)} \cU_M(\psi^{(i)}),
\end{aligned}
\\
&\label{5.10}
\begin{aligned}
(\rU_M^{(0)}+\rI_{d(M)})^{-1}\rP_{M,\bot}^{(0)}\cQ_M &(\psi^{(i)})= \rP_{M,\bot}^{(0)}\rP_{M,1}
\Big(\Pi_{\g,M}\cU_M'(\psi^{(i)})-\iu \Th_{\g,M}(0)\rP_{M,\bot}^{(0)}\cU_M(\psi^{(i)})\Big)
\\
&-\rP_{M,\bot}^{(0)}
 \Big(\Pi_{\g,M,1}\cU_M'(\psi^{(i)})+ \big(\rK_{M,\bot,1}-\iu \Th_{\g,M,1}\big)\rP_{M,\bot}^{(0)}\cU_M(\psi^{(i)})\Big).
\end{aligned}
\end{align}

Owing to identities (\ref{5.10}) and the vertex conditions for $\psi^{(i)}$, we can rewrite formula (\ref{2.25c}) for $\cL_M(\psi^{(i)})$ as
\begin{align*}
\cL_M(\psi^{(i)})=\rP_{M,\bot}^{(0)}\rP_{M,1}
\cV_M(\psi^{(i)}) -\rP_{M,\bot}^{(0)}\rK_{M,\bot,1}\rP_{M,\bot}^{(0)}
\cU_M(\psi^{(i)}).
\end{align*}
Substituting this identity and (\ref{5.18a}), (\ref{5.9}) into formula (\ref{2.25b}) for $Q_M^{(ij)}$, we obtain immediately that
\begin{equation}\label{5.22a}
\begin{aligned}
Q_M^{(ij)}=& \big(\rP_{M,1} \cV_M(\psi^{(i)}),
\rP_{M,\bot}^{(0)}\cU_M(\psi^{(j)})\big)_{\mathds{C}^{d(M)}}
-\big(\rK_{M,\bot,1}\rP_{M,\bot}^{(0)} \cU_M(\psi^{(i)}),\rP_{M,\bot}^{(0)} \cU_M(\psi^{(j)})
\big)_{\mathds{C}^{d(M)}}
\\
&+\big(\rK_{M,1}\rP_M^{(0)}\cV_M(\psi^{(i)}),
\rP_M^{(0)}\cU_M(\psi^{(j)})
\big)_{\mathds{C}^{d(M)}}
+\big(\rP_{M,1}\rP_M^{(0)}\cU_M(\psi^{(i)}), \rP_M^{(0)}\cV_M(\psi^{(j)})
\big)_{\mathds{C}^{d(M)}}.
\end{aligned}
\end{equation}
Since the matrices $\rK_{M,\bot,1}$ and $\rK_{M,1}$ are self-adjoint,
we conclude that $Q_M^{(ji)}=\overline{Q_M^{(ij)}}$. In view of formulae (\ref{2.31}), (\ref{2.25a}), this completes the proof.
\end{proof}

\section{Auxiliary operator on graph $\g$}

In this section we consider an auxiliary operator on certain extension of the graph $\g$ and we study its properties, which will be employed then in the proof of our main results. The mentioned extension is another graph denoted by $\g_{ex}$ and obtained by attaching additional unit edges $\ed_i^{ex}$, $i\in J_j$, $j=1,\ldots,n$, to the each vertex $M_j$, $j=1,\ldots,n$, in the graph $\g$. The boundary vertices, being the end-points of the edges $\ed_i^{ex}$ and not coinciding with $M_j$, are denoted by $M_i^{ex}$, $i\in J_j$, $j=1,\ldots,n$.

 The graph $\g_{ex}$ is the very similar to $\g_\infty$, the only difference is that instead of leads $\ed_i^\infty$, here we have the edges $\ed_i^{ex}$ of finite lengths. This is why, apart of $M_i^{ex}$, $i=1,\ldots,n$,
the vertices of $\g_{ex}$ coincides with the vertices in the graph $\g_\infty$. Bearing this fact in mind, in this section, while writing $M\in\g_\infty$, we shall mean that we consider a vertex in the graph $\g_{ex}$ not coinciding with $M_i^{ex}$, $i=1,\ldots,n$.

The aforementioned auxiliary operator on the graph $\g_{ex}$ is denoted by $\Op_{ex}(\e)$. Its differential expression reads as
\begin{align*}
&\hat{\Op}_{ex}(\e)u:=-\frac{d\ }{d\xi} p_\g(\cdot,\e)\frac{d u}{d\xi} +\iu\left(\frac{d\ }{d\xi}(q_\g(\,\cdot\,,\e)u)+q_\g(\,\cdot\,,\e)\frac{d u }{d\xi}\right)
+V_\g(\,\cdot\,,\e)u \quad\text{on}\quad\g,
\\
&\hat{\Op}_{ex}(\e)u:=-\sp_i(\e)\frac{d^2 u }{d\xi_i^2}+2\iu\e \sq_i(\e)\frac{d u}{d\xi_i} \quad\text{on}\quad \ed_i^{ex},\quad i\in J_j,\quad j=1,\ldots,n,
\end{align*}
where $\xi_i$ is the variable on the edge $\ed_i^{ex}$ measured from the vertex $M_j$, while $\ed_i$ are the edges in the graph $\G$ incident to the edge $M_0$.


At the vertices $M\in \g_\infty$
 we impose
 vertex conditions (\ref{5.2}).
The vertices $M_i^{ex}$ are subject to the Robin condition
\begin{equation}\label{5.4}
 \Pi_{\G,M_0}(\e)\cU'_{ex}(u)-\iu\e \Th_{\G,M_0}(\e)\cU_{ex}(u)=0,
\end{equation}
where we denote
\begin{align*}
&\Pi_{\G,M_0}(\e):=\diag\big\{\sp_i(\e) \big\}_{i=1,\ldots,d_0},\qquad \Th_{\G,M_0}(\e):=\diag\big\{\sq_i(\e) \big\}_{i=1,\ldots,d_0},
\\
&
 \cU_{ex}'(u):=\left(
\begin{aligned}
\frac{ du\big|_{\ed_1^{ex}}}{\displaystyle d\xi_1}&(M_1^{ex})
\\
&\vdots
\\
\frac{du\big|_{\ed_{d_0}^{ex}}}{d\xi_{d_0}}&(M_{d_0}^{ex})
\end{aligned}
\right),\qquad \hphantom{\cU_{ex}(u)}\cU_{ex}(u):=\left(
\begin{aligned}
u&(M_1^{ex})
\\
&\vdots
\\
u&(M_{d_0}^{ex})
\end{aligned}
\right)
\end{align*}
for all functions $u\in\dH^2(\g_{ex})$. Condition (\ref{5.4}) is obviously a particular case of condition (\ref{5.2}).

It follows from rank condition (\ref{2.9rank}), formulae (\ref{2.11b})
and the analyticity in $\e$ of the matrices $\rA_M(\e)$ and $\rB_M(\e)$ that
$\rank\big(\rA_M(\e)\ \rB_M(\e)\big)=d(M)$
for each vertex $M\in\g_{ex}$ and sufficiently small $\e$. The self-adjointness of the matrix in (\ref{2.4a}) implies the same property for the similar matrix associated with each vertex $M\in\g_{ex}$. Hence, the operator $\Op_{ex}(\e)$ satisfies all assumptions of Lemma~\ref{lm7.1} and therefore, it is self-adjoint. Then its resolvent $(\Op_{ex}(\e)-\e^2\l)^{-1}$ is well-defined for $\l\in\mathds{C}\setminus\mathds{R}$.

Our main aim in this section is to prove that  this resolvent
is meromorphic in $\e$ and
to study the structure of its pole at $\e=0$. In order to formulate a needed result, we first introduce some auxiliary notations.
The eigenvectors $\ry_i$, $i=1,\ldots,k$, of the self-adjoint matrix $\rQ$
form an orthonormalized basis in $\mathds{C}^k$ and a matrix
\begin{equation}\label{9.16}
\rY:=
\begin{pmatrix}
y_{11} & \ldots & y_{k1}
\\
\vdots & & \vdots
\\
y_{1k} & \ldots & y_{kk}
\end{pmatrix}, \qquad
\ry_i:=
\begin{pmatrix}
\ry_{i1}
\\
\vdots
\\
\ry_{ik}
\end{pmatrix},\quad i=1,\ldots,k,
\end{equation}
reduces $\rQ$ to a diagonal form:
\begin{equation}\label{4.30}
\rY^*\rQ \rY=\diag\{\l_1(\rQ),\ldots,\l_k(\rQ)\},\qquad \rY^*=\rY^{-1},
\end{equation}
where $\l_i(\rQ)$ are the eigenvalues of the matrix $\rQ$ with associated eigenvectors $\ry_i$. If zero is an eigenvalue of the matrix $\rQ$, we order the eigenvalues $\l_i(\rQ)$ so that $\l_i(\rQ)=0$ for $i=1,\ldots,k_0$. If zero is not an eigenvalue of $\rQ$, we let $k_0:=0$.

We
define
\begin{equation}\label{4.31}
\tilde{\psi}^{(j)}:=\sum\limits_{i=1}^{k} y_{ji}\psi^{(i)},\qquad \psi^{(i)}=\sum\limits_{j=1}^{k} y_{ji}\tilde{\psi}^{(j)}.
\end{equation}
The introduced functions $\tilde{\psi}^{(j)}$ are non-trivial solutions of problem~(\ref{2.20}) obeying condition~(\ref{2.23}) and
\begin{equation}\label{4.32}
\big(\cU_\g(\tilde{\psi}^{(i)}),\cU_\g(\tilde{\psi}^{(j)})\big)_{\mathds{C}^{d_0}}=\d_{ij},
\end{equation}
where $\d_{ij}$ is the Kronecker delta.

The main statement of 
this section is as follows.

\begin{lemma}\label{lmHge}
For each $\l\in\mathds{C}\setminus\mathds{R}$, the resolvent $(\Op_{ex}(\e)-\e^2\l)^{-1}
: L_2(\g_{ex})\to \dH^2(\g_{ex})$ is meromorphic for $\e$ small enough. This operator can be
represented as
\begin{equation}\label{4.81}
\big(\Op_{ex}(\e)-\e^2\l\big)^{-1} =\e^{-2}\cA_{-2}(\l)+\e^{-1}\cA_{-1}+\cA_0(\e,\l),
\end{equation}
where $\cA_0(\e,\l): L_2(\g_{ex})\to \dH^2(\g_{ex})$
is a bounded operator analytic in $\e$, while operators $\cA_{-2}$, $\cA_{-1}$ 
are given by the formulae
\begin{equation}\label{4.82}
\cA_{-2}(\l)=\sum\limits_{i=1}^{k_0} \cC^{(-2)}_i(\,\cdot\,) \tilde{\psi}^{(i)},\qquad \cA_{-1}(\l)=\sum\limits_{i=1}^{k_0} \cC^{(-2)}_i(\,\cdot\,) \hat{\psi}^{(i)} +\sum\limits_{i=1}^{k}\cC^{(-1)}_i(\,\cdot\,) \tilde{\psi}^{(i)},
\end{equation}
where $\cC^{(-2)}_i,\, \cC^{(-1)}_i: L_2(\g_{ex})\to \mathds{C}$ are some bounded linear functionals depending on $\l$ and $\hat{\psi}^{(i)}\in\dH^2(\g)$ are some functions. In particular,
\begin{equation}\label{4.83}
\begin{aligned}
&
\begin{pmatrix}
\cC^{(-2)}_{1}(f)
\\
\vdots
\\
\cC^{(-2)}_{k_0}(f)
\end{pmatrix}:=(\rX_0-\l\rG_0)^{-1}
\begin{pmatrix}
\big(f,\tilde{\psi}^{(1)}\big)_{L_2(\g_{ex})}
\\
\vdots
\\
\big(f,\tilde{\psi}^{(k_0)}\big)_{L_2(\g_{ex})}
\end{pmatrix},
\\
& \begin{pmatrix}
\cC^{(-1)}_{k_0+1}(f)
\\
\vdots
\\
\cC^{(-1)}_{k}(f)
\end{pmatrix}
:=
\rQ_1^{-1} \begin{pmatrix}
\big(f,\tilde{\psi}^{(k_0+1)}\big)_{L_2(\g_{ex})}
\\
\vdots
\\
\big(f,\tilde{\psi}^{(k)}\big)_{L_2(\g_{ex})}
\end{pmatrix}
+\rX_1(\l)
\begin{pmatrix}
\big(f,\tilde{\psi}^{(1)}\big)_{L_2(\g_{ex})}
\\
\vdots
\\
\big(f,\tilde{\psi}^{(k_0)}\big)_{L_2(\g_{ex})}
\end{pmatrix}.
\end{aligned}
\end{equation}
Here
\begin{equation*}
\rQ_1:=\diag\big\{\l_{k_0+1}(\rQ),\ldots, \l_k(\rQ)\big\},
\quad
\rG_0:=
\begin{pmatrix}
(\tilde{\psi}^{(1)}, \tilde{\psi}^{(1)})_{L_2(\g_{ex})} & \ldots & (\tilde{\psi}^{(k_0)}, \tilde{\psi}^{(1)})_{L_2(\g_{ex})}
\\
\vdots & & \vdots
\\
(\tilde{\psi}^{(1)},\tilde{\psi}^{(k_0)})_{L_2(\g_{ex})} & \ldots & (\tilde{\psi}^{(k_0)},\tilde{\psi}^{(k_0)})_{L_2(\g_{ex})}
\end{pmatrix},
\end{equation*}
 $\rX_0$ is some fixed self-adjoint matrix of size $k_0\times k_0$ and $\rX_1(\l)$ is some matrix of size $(k-k_0)\times k_0$.
\end{lemma}

\begin{remark}\label{rm6.1}
It is also possible to find explicitly the functionals $\cC^{(-1)}_{i}$ for $i=1,\ldots,k_0$, the functions $\hat{\psi}^{(i)}$ and the matrix $\rX_1(\l)$ in the above lemma, but we shall not need them in what follows. This is why we do not provide such formulae in the formulation of the lemma.
\end{remark}

The proof of Lemma~\ref{lmHge} consists of several main steps, which we present below in separate subsections.

\subsection{Meromoprhic dependence on $\e$}

In this section we begin proving Lemma~\ref{lmHge} and we shall show that the resolvent $(\Op_{ex}(\e)-\e^2\l)^{-1}$ is meromoprhic in $\e$ and has a pole at $\e=0$ of at most second order.

Since the graph $\g_{ex}$ is finite, the resolvent of the operator $\Op_{ex}(\e)$ is compact and hence, this operator has a discrete spectrum. For small $\e$ the number $\e^2\l$ is also small and this is why we should first check whether the operator $\Op_{ex}(\e)$ has small eigenvalues and if so, how they behave.
In order to understand this, we consider the operator $\Op_{ex}(0)$. We restrict the functions $\tilde{\psi}^{(j)}$ to the graph $\g_{ex}$ and we see immediately that these restrictions, still denoted by $\tilde{\psi}^{(j)}$, are the eigenfunctions of the operator $\Op_{ex}(0)$ associated with a zero eigenvalue. And vice versa, if $\psi$ is an eigenfunction of $\Op_{ex}(0)$ associated with the zero eigenvalue, then in view of the definition of the differential expression $\Op_{ex}(0)$ on the edges $\ed_i^{ex}$ and the Neumann condition at their end-points $M_i^{ex}$, the function $\psi$ is necessary constant on each of these edges. Then we replace the edges $\ed_i^{ex}$ by infinite edges $\ed_i^\infty$ passing in this way to the graph $\g_\infty$ and we continue then the function $\psi$ by the aforementioned constants on entire edges $\ed_i^\infty$. The resulting function solves problem (\ref{2.20}). Therefore, zero is an eigenvalue of the operator $\Op_{ex}(0)$ if and only if the operator $\Op_\infty$ has a virtual level at the bottom of its essential spectrum. The eigenfunctions of $\Op_{ex}(0)$ associated with the zero eigenvalues are the aforementioned restrictions $\tilde{\psi}^{(j)}$ to $\g_{ex}$.

The resolvent $(\Op_{ex}(\e)-z)^{-1}$ depends analytically in small $\e$ for $z\in\mathds{C}$ separated from the spectrum of the operator $\Op_{ex}(0)$.
In a particular case when the operator $\Op_{ex}(\e)$ is just a Schr\"odinger operator with a fixed potential independent of $\e$, this is a corollary of a more general result presented in Theorem~3.5 in \cite{BK}. For the case of the differential expression with general varying and $\e$-dependent coefficients the technique used in \cite{BK} still works as it is said in \cite{BK} after Equation (2). Hence, due to standard results in the analytic perturbation theory, the operator $\Op_{ex}(\e)$ has exactly $k$ eigenvalues, counting their multiplicity, which converge to zero as $\e\to+0$. We denote these eigenvalues by $\L_j^{(\e)}$, $j=1,\ldots,k$, while the associated eigenfunctions orthonormalized in $L_2(\g_{ex})$
are denoted by $\vp_j^{(\e)}$. According general Theorems~3.8,~3.9 in \cite{BK}, the eigenvalues $\L_j(\e)$ are analytic in $\e$ and the eigenfunctions $\vp_j$ can be also chosen being analytic in $\e$.

Standard representation for the resolvent of a self-adjoint operator, see \cite[Ch. V, Sect. 3.5]{Kato}, then gives:
\begin{equation}\label{9.1}
\big(\Op_{ex}(\e)-\e^2\l\big)^{-1}f= \sum\limits_{j=1}^{k} \frac{\big(f,\vp_j^{(\e)}\big)_{L_2(\g_{ex})}} {\L_j^{(\e)}-\e^2\l}
\vp_j^{(\e)} +\cA_1(\e,\l),
\end{equation}
and $\cA_1$ is the reduced resolvent, which can be represented as
\begin{equation*}
\cA_1(\e,\l)=\frac{1}{2\pi\iu}\int\limits_{|z|=\d} (z-\e^2\l)^{-1}
\big(\Op_{ex}(\e)-z\big)^{-1}\,dz,
\end{equation*}
where $\d$ is a sufficiently small fixed positive constant such that the ball $\{z\in \mathds{C}:\, |z|\leqslant\d\}$ contains only zero eigenvalue of the operator $\Op_{ex}(0)$ and is separated from the other eigenvalues. By the aforementioned analyticity of the resolvent $(\Op_{ex}(\e)-z)^{-1}$ in $\e$ and the known analyticity in $z$, the above formula for the operator $\cA_1$ implies that it is analytic in $\e$ for each fixed $\l\in\mathds{C}\setminus\mathds{R}$.

Since the eigenvalues $\L_j^{(\e)}$ and the eigenfunctions $\vp_j^{(\e)}$ are also analytic in $\e$, they are represented by their Taylor series as follows:
\begin{equation}\label{9.4}
\L_j^{(\e)}=\e\L_j^{(1)}+\e^2\L_j^{(2)}(\e) +\e^3\L_j^{(3)}(\e) +\e^4\tilde{\L}_j^{(\e)}, \qquad
\vp_j^{(\e)}=\vp_j^{(0)}+\e\vp_j^{(1)}+ \e^2\vp_j^{(1)} +\e^3\tilde{\vp}_j^{(\e)},
\end{equation}
where $\L_j^{(i)}$, $i=1,2$, $\vp_j^{(i)}$, $i=0,1,2$, are the leading terms in the Taylor series and $\tilde{\L}_j^{(\e)}$, $\tilde{\vp}_j^{(\e)}$ are remainders analytic in $\e$. The functions $\vp_j^{(0)}$ are eigenfunctions of the operators $\Op_{ex}(0)$ associated with the zero eigenvalue and hence, they are some linear combinations of the functions $\psi^{(j)}$.

We substitute representations (\ref{9.4}) into the terms in sum (\ref{9.1}) and expand the result into the Laurent series in $\e$. Such expansion is possible owing to the inequality $\IM\l\ne0$ and it reads as:
\begin{align}\label{9.3}
&
\begin{aligned}
\frac{\big(f,\vp_j^{(\e)}\big)_{L_2(\g_{ex})}} {\L_j^{(\e)}-\e^2\l}
\vp_j^{(\e)}=
\frac{\big(f,\vp_j^{(0)}\big)_{L_2(\g_{ex})}}{\e \L_j^{(1)}}\vp_j^{(0)} + \cA_{2,j}(\e,\l)f\quad\text{if}\quad \L_j^{(1)}\ne0,
\end{aligned}
\\
&\label{9.5}
\begin{aligned}
\frac{\big(f,\vp_j^{(\e)}\big)_{L_2(\g_{ex})}} {\L_j^{(\e)}-\e^2\l}
\vp_j^{(\e)}=& \frac{\big(f,\vp_j^{(0)}\big)_{L_2(\g_{ex})}} {\e^2 ( \L_j^{(2)}-\l)}\vp_j^{(0)}-\frac{\L_j^{(3)}\big(f, \vp_j^{(0)}\big)_{L_2(\g_{ex})}} {\e(\L_j^{(2)}-\l)^2}\vp_j^{(0)}
\\
&+ \frac{1}{\e(\L_j^{(2)}-\l)}
 \bigg(\big(f,\vp_j^{(1)}\big)_{L_2(\g_{ex})} \vp_j^{(0)}+ \big(f,\vp_j^{(0)}\big)_{L_2(\g_{ex})} \vp_j^{(1)}\bigg)
 \\
 &+ \cA_{2,j}(\e,\l)f\quad\text{if}\quad \L_j^{(1)}=0,
\end{aligned}
\end{align}
where $\cA_{2,j}(\e,\l):L_2(\g_{ex})\to\dH^2(\g)$ are some linear bounded operators analytic in $\e$ for each fixed $\l$. We observe that the denominators in the quotients in (\ref{9.5}) are non-zero since $\L_j^{(1)}$, $\L_j^{(2)}$ are real, while the imaginary part of $\l$ is non-zero.

We substitute representations (\ref{9.3}), (\ref{9.5}) into (\ref{9.1}) and we conclude that the resolvent $(\Op_{ex}(\e)-\e^2\l)^{-1}$ is meromoprhic in $\e$ and has a pole at $\e=0$ of at most second order. In order to determine the coefficients at $\e^{-2}$ and $\e^{-1}$ in the corresponding Laurent series for $(\Op_{ex}(\e)-\e^2\l)^{-1}$, we need to find $\L_j^{(1)}$, $\L_j^{(2)}$, $\vp_j^{(1)}$ and this will be done in the next subsection.

\subsection{Expansions for eigenvalues and eigenfunctions}

In this subsection we find the leading coefficients in expansions (\ref{9.4}). In order to do this, we substitute these expansions into the eigenvalue equation $\Op_{ex}(\e)\vp_j^{(\e)}=\L_j^{(\e)}\vp_j^{(\e)}$, expand the coefficients into the powers series in $\e$ and equate the coefficients at the like powers. This gives certain boundary value problems for $\vp_j^{(\e)}$ and $\L_j^{(\e)}$. Before doing this,
as in Section~\ref{secQ}, we rewrite vertex condition (\ref{5.2}) at each vertex $M\in \g_\infty$ to (\ref{4.16}), (\ref{4.17}) and we 
employ the projectors $\rP_M(\e)$ and $\rP_{M,\bot}(\e)$ and the transforming function $\rS_M(\e)$ obeying (\ref{4.13}), (\ref{3.8a}).
Then, after the above described procedure of substituting the expansions into the eigenvalue equation, we arrive at the equation
\begin{equation}\label{9.6}
\Op_{ex}(0)\vp_j^{(0)}=0
\end{equation}
and at the following boundary value problem for $\vp_j^{(1)}$:
\begin{equation}\label{9.10}
\begin{aligned}
&\hat{\Op}_{ex}(0)\vp_j^{(1)}=g_j^{(1)} +
\L_j^{(1)} \vp_j^{(1)} \quad\text{on}\quad \g_{ex},
\qquad
\Pi_{\G,M_0}(0) \cU_{ex}' (\vp_j^{(1)})= \rf_{ex,j}^{(1)},
\\
&\rP_M^{(0)} \cU_M(\vp_j^{(1)})=\rf_{M,j}^{(1)},
\qquad
\rP_{M,\bot}^{(0)} \cV_M (\vp_j^{(1)})
+
\rK_{M,\bot}(0)
\cU_M(\vp_j^{(1)})=\rf_{M,\bot,j}^{(1)}\quad\text{at}\quad M\in \g_\infty,
\end{aligned}
\end{equation}
where
\begin{equation}\label{9.12}
\begin{aligned}
&g_j^{(1)}:= -
\frac{d \hat{\Op}_{ex}(\e) \vp_j^{(i)}}{d\e}\bigg|_{\e=0},
\qquad \rf_{ex,j}^{(1)}:=
\Pi_{\G,M_0,1}^{(1)}\,\cU_{ex}(\vp_j^{(0)}),
\\
&\rf_{M,\bot,j}^{(1)}:=-\rP_{M,\bot}^{(0)}
\sum\limits_{M\in \g_\infty }\big(
\tilde{\rS}_{M,\bot,1}\, \cU_M'(\vp_j^{(0)})+\tilde{\rK}_{M,\bot,1}\, \cU_M(\vp_j^{(0)})\big),
\\
&\rf_{M,j}^{(1)}:=-\rP_M^{(0)}
\sum\limits_{M\in\g_\infty}
\big(\tilde{\rS}_{M,1}\,\cU_M(\vp_j^{(0)}) + \tilde{\rK}_{M,1} \cU_M'(\vp_j^{(0)})\big).
\end{aligned}
\end{equation}

Equation (\ref{9.6}) is solved immediately:
\begin{equation}\label{9.13}
\vp_j^{(0)}=\sum\limits_{i=1}^{k} \a_{ji} \tilde{\psi}^{(i)},
\end{equation}
where $\a_{ji}$ are some constants such that the functions $\vp_j^{(0)}$ are orthonormalized in $L_2(\g_{ex})$. This condition means that the matrices
\begin{equation*}
\a:=
\begin{pmatrix}
\a_{11} & \ldots & \a_{k1}
\\
\vdots & & \vdots
\\
\a_{1k} & \ldots & \a_{kk}
\end{pmatrix},\qquad \rG:=
\begin{pmatrix}
(\tilde{\psi}^{(1)}, \tilde{\psi}^{(1)})_{L_2(\g_{ex})} & \ldots & (\tilde{\psi}^{(k)}, \tilde{\psi}^{(1)})_{L_2(\g_{ex})}
\\
\vdots & & \vdots
\\
(\tilde{\psi}^{(k)},\tilde{\psi}^{(1)})_{L_2(\g_{ex})} & \ldots & (\tilde{\psi}^{(k)}, \tilde{\psi}^{(k)})_{L_2(\g_{ex})}
\end{pmatrix}
\end{equation*}
should satisfy the equation $\rG\a\a^*=\a^*\rG\a=\rI_k$.

In order to analyze problem (\ref{9.10}), we shall employ the following auxiliary lemma.

\begin{lemma}\label{lmHg1}
Assume that Condition~\ref{C1} holds.
 Given an arbitrary family of vectors $\rf_M\in \rP_M^{(0)} \mathds{C}^{d(M)}$, $\rf_{M,\bot}\in \rP_{M,\bot} \mathds{C}^{d(M)}$ for each vertex $M\in\g_\infty$, an arbitrary vector
$\rf_{ex}\in \mathds{C}^{d_0}$, and an arbitrary function $g\in L_2(\g_{ex})$,
the boundary value problem
\begin{equation}\label{5.5a}
\begin{aligned}
&\hat{\Op}_{ex}(0) u=g \quad\text{on}\quad \g_{ex},
\qquad
\Pi_{\G,M_0}(0)\cU_{ex}'(u)=\rf_{ex},
\\
&\rP_M^{(0)} \cU_M(u)=\rf_M,
\qquad
\rP_{M,\bot}^{(0)}\cV_M(u)
+
\rK_{M,\bot}(0)
\cU_M(u)=\rf_{M,\bot}\quad\text{at}\quad M\in\g_\infty,
\end{aligned}
\end{equation}
is solvable in $\dH^2(\g_{ex})$ if and only if
\begin{equation}\label{5.5b}
\begin{aligned}
(g,\tilde{\psi}^{(j)})_{L_2(\g_{ex})}=& -\big(\rf_{ex},\cU_\g(\tilde{\psi}^{(j)})\big)_{\mathds{C}^{d_0}} +\sum\limits_{M\in\g_\infty} \big(\rf_{M,\bot},
\cU_M(\tilde{\psi}^{(j)})\big)_{\mathds{C}^{d(M)}}
-\sum\limits_{M\in\g_\infty} \big(\rf_M,
\cV_M(\tilde{\psi}^{(j)})
\big)_{\mathds{C}^{d(M)}}
\end{aligned}
\end{equation}
for each $j=1,\ldots,k$.
\end{lemma}

\begin{proof} Throughout the proof by $C$ we denote various inessential constants independent of $g$, $\rf_{ex}$, $\rf_M$, $\rf_{M,\bot}$.
Let $u_0$ be a function on $\g_{ex}$ satisfying the boundary conditions in (\ref{5.5a}). Such function can be easily constructed on each edge as a linear function with appropriate coefficients multiplied by a cut-off function.
We seek a solution to problem (\ref{5.5a})
as $u=u_0+\tilde{u}$
 and for a new unknown function $\tilde{u}$ we obtain the equation
\begin{equation}\label{4.25}
\Op_{ex}(0)\tilde{u}=\tilde{g},\qquad \tilde{g}:=g-\hat{\Op}_{ex}(0)u_0
\end{equation}

Since zero is an eigenvalue of the operator $\Op_{ex}(0)$, equation (\ref{4.25}) is solvable if and only if its right hand side orthogonal to all eigenfunctions associated with the zero eigenvalue:
\begin{equation}\label{4.26}
(\tilde{g},\psi^{(j)})_{L_2(\g_{ex})}=0,\qquad j=1,\ldots,k.
\end{equation}
Integrating by parts and employing the definition of the functions $\psi^{(j)}$ and the vertex conditions for the function $u_0$, it is straightforward to confirm that
\begin{align*}
\big(\hat{\Op}_{ex}(0)u_0, \tilde{\psi}^{(j)}\big)_{L_2(\g_{ex})} =&-\sum\limits_{i=1}^{d_0} \sp_i(0)\frac{d u_0\big|_{\ed_i}}{d\xi_i}(M_i^{ex}) \overline{\tilde{\psi}^{(j)}}\big|_{\ed_i}(M_i^{ex})
 \\
 &+ \sum\limits_{M\in\g_\infty} \big(\cV_M(u_M),
 \cU_M(\tilde{\psi}^{(j)})\big)_{\mathds{C}^{d(M)}}
 - \sum\limits_{M\in\g_\infty} \big(\cU_M(u_M), \cV_M (\tilde{\psi}^{(j)})
 \big)_{\mathds{C}^{d(M)}}
\\
=&-\big(\rf_{ex},\cU_\g(\psi^{(j)})\big)_{\mathds{C}^{d_0}}
+\sum\limits_{M\in\g} \big(\rf_{M,\bot},
\cU_M(\tilde{\psi}^{(j)})\big)_{\mathds{C}^{d(M)}}
\\
&-\sum\limits_{M\in\g_\infty} \big(\rf_M,
\cV_M(\tilde{\psi}^{(j)})
\big)_{\mathds{C}^{d(M)}}.
\end{align*}
Substituting these identities into (\ref{4.26}), we conclude that equation (\ref{4.25}), and hence, problem (\ref{5.5a}) is solvable if and only if conditions (\ref{5.5b}) are satisfied.
The proof is complete.
\end{proof}

In order to check the solvability of problem (\ref{9.10}),
we should substitute formulae (\ref{9.13}) into (\ref{9.12}) and calculate then the terms in solvability condition (\ref{5.5b}). To avoid unnecessary technicalities, we first do such calculations for a particular case, when instead of $\vp_j^{(0)}$, we substitute $\tilde{\psi}^{(j)}$ into (\ref{9.12}).
Once we obtain solvability conditions (\ref{5.5b}) for such particular case, the needed condition is just a their linear combinations with coefficients $\a_{ij}$.

For $j=1,\ldots,k$ we denote
\begin{equation}\label{4.34}
\begin{aligned}
&
{g}_j=-\frac{d \hat{\Op}_{ex}(\e)\tilde{\psi}^{(j)} }{d\e}\bigg|_{\e=0},
\qquad &&{\rf}_{M,j} =-\rP_M^{(0)} \big(\tilde{\rS}_{M,1}\,\cU_M(\tilde{\psi}^{(j)})
+\tilde{\rK}_{M,1}\,\cU_M'(\tilde{\psi}^{(j)})\big),
\\
&
{\rf}_{ex,j}=
\iu\Th_{\G,M_0}(0) \cU_{ex} (\tilde{\psi}^{(j)}),
\qquad &&
{\rf}_{M,\bot,j} =-\rP_{M,\bot}^{(0)}
\big(\tilde{\rS}_{M,\bot,1}\,\cU_M'(\tilde{\psi}^{(j)}) +\tilde{\rK}_{M,\bot,1}\,\cU_M(\tilde{\psi}^{(j)})\big).
\end{aligned}
\end{equation}
We observe that
\begin{equation}\label{9.15}
\begin{aligned}
&g_j^{(1)}=\sum\limits_{i=1}^{k}\a_{ji}g_i,\qquad
\hphantom{\rf_{ex}}
{\rf}_{ex,j}^{(1)}= \sum\limits_{i=1}^{k}\a_{ji}{\rf}_{ex,i},
\\
 &{\rf}_{M,j}^{(1)}= \sum\limits_{i=1}^{k}\a_{ji}{\rf}_{M,i},\qquad {\rf}_{M,\bot,j}^{(1)}= \sum\limits_{i=1}^{k}\a_{ji}{\rf}_{M,\bot,i}.
\end{aligned}
\end{equation}

Integrating by parts and taking into consideration the vertex conditions for $\psi^{(i)}$, we obtain:
\begin{equation}\label{5.13}
\begin{aligned}
-({g}_j, \tilde{\psi}^{(l)})_{L_2({\g_{ex}})}= &\left(\frac{d p_\g}{d\e}(\,\cdot\,,0) \frac{d\tilde{\psi}^{(j)}}{d\xi}, \frac{d\tilde{\psi}^{(l)}}{d\xi}\right)_{L_2(\g)} +
\left(\frac{d\tilde{\psi}^{(l)}}{d\xi},\iu \frac{d q_\g}{d\e}(\,\cdot\,,0) \tilde{\psi}^{(l)}\right)_{L_2(\g)}
\\
&+\left(\iu \frac{d q_\g}{d\e}(\,\cdot\,,0) \tilde{\psi}^{(j)}, \frac{d\tilde{\psi}^{(l)}}{d\xi}\right)_{L_2(\g)}
+\left(\frac{d V_\g}{d\e}(\,\cdot\,,0) \tilde{\psi}^{(j)},\tilde{\psi}^{(l)}\right)_{L_2(\g)}
\\
&+\sum\limits_{M\in\g_\infty}
\Big(
\Pi_{\g,M,1}\,\cU_M'(\tilde{\psi}^{(j)})-\iu \Th_{\g,M,1}\,\cU_M(\tilde{\psi}^{(j)}),
\cU_M(\tilde{\psi}^{(l)})\Big)_{\mathds{C}^{d(M)}}
\\
&-\iu\big(\Th_{\G,M_0}\,\cU_\g(\tilde{\psi}^{(j)}), \cU_\g(\tilde{\psi}^{(l)})\big)_{\mathds{C}^{d_0}}.
\end{aligned}
\end{equation}
The definition of the matrices $\tilde{\rS}_{M,\bot}(\e)$, $\tilde{\rK}_{M,\bot}(\e)$ in (\ref{4.17}), the definition of $\rK_{M,\bot}(0)$, the first identity in (\ref{4.13}) and identities (\ref{5.8}), (\ref{4.63}) imply that
\begin{align*}
& \tilde{\rS}_{M,1}=\rS_{M,1}-\iu \rK_{M,1} \Th_{\g,M,1}, && \tilde{\rK}_{M,1}=\rP_M^{(0)} \rK_{M,1} \Pi_{\g,M}(0),
\\
&\tilde{\rS}_{M,\bot,1}=\Pi_{\g,M,1} + \rS_{M,1}\Pi_{\g,M}(0),\quad && \tilde{\rK}_{M,\bot,1}=\rP_{M,\bot}^{(0)}\big(\rK_{M,\bot,1}-\iu\Th_{\g,M,1} + \iu\rP_{M,1} \Th_{\g,M}(0)\big).
\end{align*}
We substitute the latter formulae and (\ref{5.13}), (\ref{4.34}) into the right hand side in (\ref{4.37}) and use identities (\ref{5.22a}), (\ref{4.31}), (\ref{4.32}). Comparing then the result with (\ref{2.31}), (\ref{2.25a}), (\ref{2.25b}) and employing Lemma~\ref{lm7.2}, we see that
\begin{equation}\label{4.37}
\begin{aligned}
\l_j(\rQ)\d_{jl}=& -({g}_j,\tilde{\psi}^{(l)})_{L_2(\g_{ex})} - \big({\rf}_{ex,j}, \cU_\g(\tilde{\psi}^{(l)})\big)_{\mathds{C}^{d_0}}
\\
&+\sum\limits_{M\in\g_\infty} \big({\rf}_{M,\bot,j}, \cU_M(\tilde{\psi}^{(l)})\big)_{\mathds{C}^{d(M)}}
- \sum\limits_{M\in\g_\infty} \big({\rf}_{M,j},\cV_M(\tilde{\psi}^{(l)})
\big)_{\mathds{C}^{d(M)}}.
\end{aligned}
\end{equation}
It follows from these identities and (\ref{9.15}) that solvability conditions (\ref{5.5b}) for problem (\ref{9.10}) can be written as
\begin{equation*}
\a_{jl} \l_l(\rQ) =\L_j^{(1)} \sum\limits_{i=1}^{k} (\psi^{(i)},\psi^{(l)})_{L_2(\g_{ex})}\a_{ji},\qquad j,l=1,\ldots,k.
\end{equation*}
The above identities mean that $\L_j^{(1)}$ and the vectors $\a_j:=\big(\a_{j1}\ \ldots \ \a_{jk}\big)^t$ are the eigenvalues and the associated eigenvectors of the following equivalent eigenvalue equations:
\begin{equation}\label{9.17}
\begin{aligned}
&\diag\big\{\l_1(\rQ),\ldots,\l_k(\rQ) \big\}\a_j=\L_j^{(1)}\rG\a_j,
\\
& \big(\rG^{-\frac{1}{2}}\diag \big\{\l_1(\rQ),\ldots,\l_k(\rQ) \big\}
\rG^{-\frac{1}{2}}-\L_j^{(1)}\big)
\rG^{\frac{1}{2}}\a_j=0.
\end{aligned}
\end{equation}
Here we have also employed the fact that the Gram matrix $\rG$ is self-adjoint and positive definite that ensures the existence of $\rG^{\pm \frac{1}{2}}$. The matrix in the second equation in (\ref{9.17}) is self-adjoint and hence, its eigenvalues $\L_j^{(1)}$ are real and the associated eigenvectors can be chosen orthonormalized in $\mathds{C}^k$. These eigenvectors produce the eigenvectors of the first equation in (\ref{9.17}) via multiplication by $\rG^{-\frac{1}{2}}$.
We also observe that $\L_j^{(1)}=\l_j(\rQ)=0$, $j=1,\ldots,k_0$. In view of the aforementioned condition $\a^*\rG\a =\rI_k$ and the first equation in (\ref{9.17}) we then get easily that
\begin{equation}\label{9.21}
\begin{aligned}
&\a_{ji}=0,\hphantom{(\rG)_{ji}}\qquad j=1,\ldots,k_0,\qquad i=k_0+1,\ldots,k,
\\
& (\rG\a)_{ji}=0,\qquad
j=k_0+1,\ldots,k,\qquad i=1,\ldots,k_0,
\end{aligned}
\end{equation}
where $(\rG\a)_{ji}$ is the entry of the matrix $\rG\a$ in the $i$th row and $j$th column. We also observe that the first equation in (\ref{9.17}) can be rewritten for the matrix $\a$ as
\begin{equation}\label{9.27}
\diag\big\{\l_1(\rQ),\ldots,\l_k(\rQ) \big\}\a= \rG\a \diag\big\{\L_1^{(1)}, \ldots, \L_k^{(1)}\big\}.
\end{equation}

\subsection{Structure of pole}

In accordance with representation (\ref{9.1}), the singularity of the operator $\big(\Op_{ex}(\e)-\e^2\l\big)^{-1}$ at $\e=0$ is determined by the sums over $j$ of the corresponding singular terms in (\ref{9.3}), (\ref{9.5}). Let us find these sums.

The condition $\L_j^{(1)}\ne0$ in (\ref{9.3}) is obviously satisfied only for $j\geqslant k_0+1$.
By formula (\ref{9.13}) we get
\begin{equation}\label{9.22}
\sum\limits_{j=k_0+1}^{k} \frac{\big(f,\vp_j^{(0)}\big)_{L_2(\g_{ex})}}{ \L_j^{(1)}}\vp_j^{(0)}=\sum\limits_{j=k_0+1}^{k}
\sum\limits_{i,l=1}^{k} \frac{\a_{ji}\overline{\a_{jl}}}{\L_j^{(1)}} (f,\tilde{\psi}^{(l)})_{L_2(\g_{ex})}\tilde{\psi}^{(i)}.
\end{equation}
The numbers $\sum\limits_{j=k_0+1}^{k}\frac{\a_{ji}\overline{\a_{jl}}}{\L_j^{(1)}}$ are the entries of the following matrix product
\begin{equation}\label{9.23}
\begin{pmatrix}
\a_{k_0+1 1} & \ldots & \a_{k1}
\\
\vdots && \vdots
\\
\a_{k_0+1 k} & \ldots & \a_{k k}
\end{pmatrix} \diag\big\{(\L_{k_0+1}^{(1)})^{-1},\ldots, (\L_k^{(1)})^{-1}\big\}
\begin{pmatrix}
\overline{\a_{k_0+1 1}} & \ldots &\overline{\a_{k_0+1 k}}
\\
\vdots && \vdots
\\
\overline{\a_{k1}} & \ldots & \overline{\a_{k k}}
\end{pmatrix}.
\end{equation}
Employing equation (\ref{9.27}) and the aforementioned orthonormality condition $\rG\a\a^*=\rI_k$, it is straightforward to confirm that matrix product (\ref{9.23}) is equal to $\diag\big\{\l_{k_0+1}^{-1}(\rQ),\ldots, \l_k^{-1}(\rQ)\big\}$. Hence, identity (\ref{9.22}) can be continued as
\begin{equation}\label{9.24}
\sum\limits_{j=k_0+1}^{k} \frac{\big(f,\vp_j^{(0)}\big)_{L_2(\g_{ex})}}{ \L_j^{(1)}}\vp_j^{(0)}=\sum\limits_{j=k_0+1}^{k}
 \frac{ (f,\tilde{\psi}^{(j)})_{L_2(\g_{ex})}}{\L_j^{(1)}} \tilde{\psi}^{(j)}.
\end{equation}

A similar sum comes from (\ref{9.5}):
\begin{equation}\label{9.25}
\sum\limits_{j=1}^{k_0} \frac{\big(f,\vp_j^{(0)}\big)_{L_2(\g_{ex})}} { \L_j^{(2)}-\l}\vp_j^{(0)}=\sum\limits_{j=1}^{k_0} \sum\limits_{i,l=1}^{k_0} \frac{\a_{ji}\overline{\a_{jl}}}{\L_j^{(2)}-\l} (f,\tilde{\psi}^{(l)})_{L_2(\g_{ex})}\tilde{\psi}^{(i)},
\end{equation}
where we have employed the first identities from (\ref{9.21}).
Here the numbers $\sum\limits_{j=1}^{k_0} \frac{\a_{ji}\overline{\a_{jl}}}{\L_j^{(2)}-\l}$ can be regarded as the entries of the following matrix product
\begin{equation*}
\b
\diag\big\{(\L_1^{(2)}-\l)^{-1},\ldots, (\L_{k_0}^{(2)}-\l)^{-1}\big\} \b^*,\quad \text{where}\quad \b:=\begin{pmatrix}
\a_{1 1} & \ldots & \a_{k_0 1}
\\
\vdots && \vdots
\\
\a_{1 k_0} & \ldots & \a_{k_0 k_0}
\end{pmatrix}.
\end{equation*}
We note that the numbers $\L_j^{(2)}$ are real, while the imaginary part of $\l$ is non-zero and this is why the diagonal matrix in the above product is well-defined.
In view of the orthonormalization condition $\a^*\rG\a=\rI_k$ and the first identities in (\ref{9.21}) we have $\b^*\rG_0\b=\rI_{k_0}$, $\rG_0=(\b^{-1})^*\b^{-1}$ and hence,
\begin{align*}
\b
\diag&\big\{(\L_1^{(2)}-\l)^{-1},\ldots, (\L_{k_0}^{(2)}-\l)^{-1}\big\} \b^* =\Big((\b^{-1})^*
\diag\big\{\L_1^{(2)}-\l,\ldots, \L_{k_0}^{(2)}-\l\big\} \b^{-1}
\Big)^{-1}
\\
=& \Big((\b^{-1})^*
\diag\big\{\L_1^{(2)},\ldots, \L_{k_0}^{(2)}\big\} \b^{-1}-\l \rG_0
\Big)^{-1}
\end{align*}
and the matrix $\rX_0:=(\b^{-1})^*
\diag\big\{\L_1^{(2)},\ldots, \L_{k_0}^{(2)}\big\} \b^{-1}$ is obviously self-adjoint. Therefore, it follows from (\ref{9.25}) that
\begin{equation*}
\sum\limits_{j=1}^{k_0} \frac{\big(f,\vp_j^{(0)}\big)_{L_2(\g_{ex})}} { \L_j^{(2)}-\l}\vp_j^{(0)}= \sum\limits_{i=1}^{k_0} \cC^{(-2)}_i(f) \tilde{\psi}^{(i)},
\end{equation*}
where the functionals $\cC^{(-2)}_i$ are defined in (\ref{4.83}). Substituting the obtained identity, (\ref{9.24}) into (\ref{9.3}), (\ref{9.5}) and emplyoing then representation (\ref{9.1}), we arrive at representation (\ref{4.81}), (\ref{4.82}), (\ref{4.83}). This completes the proof of Lemma~\ref{lmHge}.

\section{Analyticity of resolvent}\label{ss:Th1}

In this section we prove Theorem~\ref{th1}. Our main idea is to reduce the equation for the resolvent, that is, $(\Op_\e-\l)u_\e=f$, to a linear system of equations and to analyze then the dependence of its solution on $\e$. Since our graph $\G_\e$ consists of two subgraphs $\G$ and $\g_\e$ containing respectively finite and small edges, it is natural to consider the function $u_\e$ as a solution to appropriate boundary value problems on these two graphs and then to match these solutions on the edges $\ed_i$, $i\in J_j$, $j=1,\ldots,n$, incident to the vertices $M_j$.

Let $f_\G\in L_2(\G)$, $f_\g\in L_2(\g)$ be two arbitrary functions. In terms of these functions we introduce one more function $f\in L_2(\G_\e)$ as $f:=f_\G$ on $\G$ and $f:=\cS_\e f_\g$ on $\g_\e$.
We let $u_\e:=(\Op_\e-\l)^{-1}$ and we see immediately that
the restriction of the function $u_\e$ to the graph $\G$, that is, the function $W_\e:=\cP_\G u_\e$, solves the boundary value problem for the equation
\begin{equation}\label{3.13}
(\hat{\Op}(\e)-\l)W_\e=f_\G\quad\text{on}\quad\G
\end{equation}
subject to homogeneous vertex conditions (\ref{2.9}) at all vertices $M\in\G$ except for $M_0$, while at the vertex $M_0$ an inhomogeneous vertex condition
\begin{equation}\label{3.14}
\cU_{M_0}(W_\e)=\ra(\e),\qquad \ra(\e):=
\begin{pmatrix}
a_1(\e)
\\
\vdots
\\
a_{d_0}(\e)
\end{pmatrix}
\end{equation}
holds with some constants $a_i=a_i(\e)$, $i=1,\ldots,d_0$. These constants are, of course, given by the left hand side in (\ref{3.14}), but we treat them as unknown. If we find them, then we can recover the function $u_\e$ on the subgraph $\G$ as a solution to the above boundary value problem.

Similarly, we consider the restriction $\cP_{\g_\e} u_\e$ of the function $u_\e$ to the graph $\g_\e$ and we rescale it by means of $\cS_\e$ letting $w_\e:=\cS_\e^{-1}\cP_{\g_\e} u_\e$. The function $w_\e$ solves the equation $(\hat{\Op}_{ex}(\e)-\e^2\l) w_\e=\e^2 f_\g$ on $\g$ and satisfies vertex conditions (\ref{5.2}). Our next aim is to continue the function $w_\e$ on the edges $\ed_i^{ex}$ so that the continuation solves certain boundary value problem again written in terms of $\hat{\Op}_{ex}(\e)$ and vertex condition (\ref{5.4}). In view of the latter vertex condition,
we denote
\begin{equation*}
a'_i(\e):=
\sp_i(\e)\frac{d w_\e\big|_{\ed_i}}{d\xi_i}(M_j)-\iu\e \sq_i(\e)w_\e\big|_{\ed_i}(M_j),\qquad \ra'(\e):=\begin{pmatrix}
a_1'(\e)
\\
\vdots
\\
a_{d_0}'(\e)
\end{pmatrix},
\end{equation*}
where $w_\e\big|_{\ed_i}$ is the restriction of $w_\e$ to an edge $\ed_i$, $i\in J_j$, incident to a vertex $M_j$, $j=1,\ldots,n$. We also observe that $w_\e\big|_{\ed_i}(M_j)=a_i(\e)$, where $a_i$ are exactly the constants introduced in (\ref{3.14}).

We continue the function $w_\e$ on the edges $\ed_i^{ex}$ as follows:
\begin{equation*}
w_\e(x_i):=a_i(\e)\phi_i (\xi_i,\e)+ a_i'(\e)\phi_i'(\xi_i,\e),
\end{equation*}
where $\xi_i$ is a variable on the edge $\ed_i^{ex}$
 and
\begin{align*}
&\phi_i(t,\e):=
e^{\iu \frac{\e\sq_i(\e)}{\sp_i(\e)}t}
\cos \frac{\e\tau_i(\e)}{\sp_i(\e)}t,
\qquad
\phi_i'(t,\e):= \frac{e^{\iu\frac{\e\sq_i(\e)}{\sp_i(\e)}t}}
{\e\tau_i(\e)}\sin \frac{\e\tau_i(\e)}{\sp_i(\e)}t,
\\
&\tau_i(\e):= \sqrt{\l\sp_i(\e)+ \sq_i^2(\e)},
\end{align*}
where the branch of the square root is fixed arbitrarily. Then it is straightforward to confirm that after such continuation the function $w_\e$
 solves
a desired boundary value problem
\begin{equation}\label{3.20}
\begin{gathered}
(\hat{\Op}_{ex}(\e)-\e^2\l) w_\e
=\e^2 \chi_\g f_\g\quad\text{on}\quad \g_{ex},
\\
\Pi_{\G,M_0}\cU_{ex}'(w_\e)-\iu\e\Th_{\G,M_0} \cU_{ex}(w_\e)
=\rL_{\exp}
(\e) \big(\rL_{\cos}(\e)\ra'(\e)-\e\rL_{\tau}(\e) \rL_{\sin}(\e) \ra(\e)\big),
\end{gathered}
\end{equation}
with vertex conditions (\ref{5.2}), where $\chi_\g$ is the characteristic function of the graph $\g$, that is, $\chi_\g=1$ on $\g$ and $\chi_\g=0$ on $\ed_i^{ex}$, $i=1,\ldots,d_0$, and
\begin{align*}
&\rL_{\exp}(\e):=\diag \left\{\exp\left(\iu\frac{\e\sq_i(\e)} {\sp_i(\e)}\right)\right\}_{i=1,\ldots,d_0}, &&
\rL_{\cos}(\e):=\diag \left\{\cos \frac{\e\tau_i(\e)} {\sp_i(\e)}\right\}_{i=1,\ldots,d_0},
\\
&\rL_{\tau}(\e):=\diag \big\{\tau_i(\e)\big\}_{i=1,\ldots,d_0}, &&
\rL_{\sin}(\e):=\diag \left\{\sin \frac{\e\tau_i(\e)} {\sp_i(\e)}\right\}_{i=1,\ldots,d_0}.
\end{align*}

Since the functions $W_\e$ and $w_\e$ are obtained from the restrictions of the same function $u_\e$ to the graphs $\G$ and $\g_\e$, certain continuity conditions at the vertices are to be satisfied. Namely, for each vertex $M_j$, $j=1,\ldots,n$, and each incident edge $\ed_i$, $i\in J_j$, the restrictions of the functions $W_\e$ and $w_\e$ to the edge $\ed_i$
should have coincinding values at $M_j$ and a similar condition should hold for their derivatives. In view of formula (\ref{3.39}) this condition can be equivalently rewritten as
\begin{equation}\label{9.28}
\begin{gathered}
W_\e\big|_{\ed_i}(M_j)=w_\e\big|_{\ed_i}(M_j),
\\
\sp_i(\e)\frac{d W_\e\big|_{\ed_i}}{dx_i}(M_j)-\iu \sq_i(\e)W_\e\big|_{\ed_i}(M_j)
=\e^{-1} \left(\sp_i(\e)\frac{d w_\e\big|_{\ed_i^{ex}}}{d\xi_i}(M_j) -\iu\e\sq_i(\e) w_\e\big|_{\ed_i^{ex}}(M_j)
\right)
\end{gathered}
\end{equation}
for all $i\in J_j$, $j=1,\ldots,n$. Both functions $W_\e$ and $w_\e$ are to be treated as solutions to boundary values problems (\ref{3.13}), (\ref{2.9}), (\ref{3.14}) and (\ref{3.20}), (\ref{5.2}). Their solutions are determined by the constants $a_i$, $a_i'$ and by the functions $f_\G$, $f_\g$. It is clear that once we find explicitly the dependence of these solutions on $a_i$ and $a_i'$ and substitute the result into (\ref{9.28}), we end up with an inhomogeneous linear system of equations for the vectors $\ra(\e)$ and $\ra'(\e)$. If we succeed then to solve it and to prove that the solution is analytic in $\e$, then we shall recover the function $W_\e$ and $w_\e$ as the solutions to the above problems and this will prove Theorem~\ref{th1}.

As a first step in the above plan of the proof, we need to analyze the structure of the solutions to problems (\ref{3.13}), (\ref{2.9}), (\ref{3.14}) and (\ref{3.20}), (\ref{5.2}) and this will be done in the following subsection.

\subsection{Analysis of problems for $W_\e$ and $w_\e$}

We first introduce   an auxiliary operator  on the graph $\G$ corresponding to the boundary value problem for $W_\e$. This is the operator with the differential expression $\hat{\Op}(\e)$ considered on $\G$
subject to vertex conditions (\ref{2.9}) at vertices $M\in\G$, $M\ne M_0$ and to the Dirichlet condition
\begin{equation}\label{4.7a}
\cU_{M_0}(u)=0\quad\text{at vertex}\quad M_0.
\end{equation}
We denote such operator by $\Op_\G(\e)$. By
Lemma~\ref{lm7.1}  the operator $\Op_\G$ is self-adjoint in $L_2(\G)$ on the domain formed by the functions from $\dH^2(\G)$
satisfying the imposed vertex conditions.
The next lemma describes its dependence on $\e$.

\begin{lemma}\label{lmHGe}
For each $\l\in\mathds{C}$ with $\IM\l\ne0$ the resolvent $(\Op_\G(\e)
-\l)^{-1}$
is well-defined for all sufficient small $\e$ and is analytic in $\e$ as an operator from $L_2(\G)$ into $\dH^2(\G)$.
\end{lemma}

This lemma is implied by Theorem~3.5 in \cite{BK} extended to our general operator; such extension is possible thanks to the remark after Equation (2) in \cite{BK}.

In order to treat inhomogeneous boundary condition (\ref{3.14}), we consider
 boundary value problems for the equation
\begin{equation}\label{3.15}
(\hat{\Op}(\e)-\l)v_{\G,i}^{(\e)}=0 \quad\text{on}\quad\G,\qquad i=1,\ldots,d_0,
\end{equation}
subject to homogeneous vertex conditions (\ref{2.9}) at all vertices $M\in\G$ except for $M_0$ and to an inhomogeneous condition
\begin{equation}\label{3.16}
\cU_{M_0}(v_{\G,i}^{(\e)})=
\begin{pmatrix}
0 & \ldots & 1 &\ldots & 0
\end{pmatrix}^t,\qquad i=1,\ldots,d_0, \quad \text{at vertex}\quad M_0,
\end{equation}
where in the vector in the right hand side one stands only at $i$th position. Similar to the proof of Lemma~\ref{lmHg1}, problem (\ref{3.15}), (\ref{2.9}), (\ref{3.16}) can be transformed to one with an inhomogeneous equation and the homogeneous boundary condition. Then we apply Lemma~\ref{lmHGe} and we see that problems (\ref{3.15}), (\ref{2.9}), (\ref{3.16}) are uniquely solvable and their solutions $v_{\G,i}^{(\e)}$ are analytic in $\e$ in the sense of the norm in the space $\dH^2(\g_{ex})$.

In view of the
boundary value problem (\ref{3.13}), (\ref{2.9}), (\ref{3.14}), we conclude that
\begin{equation}\label{3.17}
\cR_\G(\e,\l)(f_\G,f_\g)=W_\e =(\Op_\G(\e)-\l)^{-1}f_\G+\sum\limits_{i=1}^{d_0}
a_i(\e)
v_{\G,i}^{(\e)}.
\end{equation}

In order to solve problem (\ref{3.20}), (\ref{5.2}), we introduce auxiliary boundary value problems on the graph $\g_{ex}$:
\begin{equation}\label{3.21}
(\hat{\Op}_{ex}(\e)-\e^2\l)v_{\g,\e}^{(i)}=0\quad\text{on}\quad \g_{ex},\qquad
\Pi_{\G,M_0}(\e)\cU'_{ex}(v_{\g,\e}^{(i)})-\iu\e \Th_{\G,M_0}(\e)\cU_{ex} (v_{\g,\e}^{(i)})
=\tilde{\varPsi}^{(i)},
\end{equation}
where $i=1,\ldots,d_0$,
with vertex conditions (\ref{5.2}), where
\begin{equation*}
\tilde{\varPsi}^{(i)}:=\sum\limits_{i=1}^{k} y_{ij}\varPsi^{(j)}=
\cU_\g(\tilde{\psi}^{(j)}),\quad i=1,\ldots,k,\qquad \tilde{\varPsi}^{(i)}:=\varPsi^{(i)},\quad i=k+1,\ldots,d_0.
\end{equation*}

 The solvability of these problems and the dependence of the solutions on the parameter $\e$ are described in the following lemma.

\begin{lemma}\label{lm4.9}
For each $\l\in\mathds{C}\setminus\mathds{R}$
problems (\ref{3.21}), (\ref{5.2}) are uniquely solvable provided $\e$ is small enough. The solutions are meromorphic in $\e$ and satisfy the following representations:
\begin{align}\label{4.87a}
&v_{\g,\e}^{(i)}= \e^{-2}\sum\limits_{j=1}^{k_0}
\big((\rX_0-\l\rG_0)^{-1}\big)_{ji}
\tilde{\psi}^{(j)} +\e^{-1}
\tilde{v}_{\g,\e}^{(i)},\qquad i=1,\ldots,k_0,
\\
&v_{\g,\e}^{(i)}=\e^{-1} \bigg(\l_i^{-1}(\rQ) \tilde{\psi}^{(i)}  +\sum\limits_{j=1}^{k_0} \Ups_{ji}
\tilde{\psi}^{(j)} \bigg)
+\tilde{v}_{\g,\e}^{(i)}, \qquad i=k_0+1,\ldots,k,\label{4.87b}
\\
&v_{\g,\e}^{(i)}=\e^{-1}\sum\limits_{j=1}^{k_0} \Ups_{ji} \tilde{\psi}^{(j)}
+\tilde{v}_{\g,\e}^{(i)}, \qquad i=k+1,\ldots,d_0.\label{4.87c}
\end{align}
Here the functions $\tilde{v}_{\g,\e}^{(i)}$ are analytic in $\e$ in $\dH^2(\g)$-norm, $\big((\rX_0-\l\rG_0)^{-1}\big)_{ji}$ is the entry of the matrix $(\rX_0-\l\rG_0)^{-1}$ at $j$th column and $i$th row, while $\Ups_{ji}$ are some constants.
\end{lemma}

\begin{proof} As in the proof of Lemma~\ref{lmHg1}, we first construct a function $v_{bnd}^{(i)}\in \dC^2(\g)$ satisfying vertex conditions in (\ref{5.2}) and (\ref{3.21}); in addition we assume that $\cU_{ex}(v_i^{bnd})=0$.
Then we 
let $v_{\g,i}^{(\e)}=v_{bnd}^{(i)}+\tilde{v}_i$ and for a new unknown function $\tilde{v}_i$ we get the equation
\begin{equation*}
\big(\Op_{ex}(\e)-\e^2\l\big) \tilde{v}_i=g,\qquad
g:=
-\big(\hat{\Op}_{ex}(\e)-\e^2\l\big) v_i^{bnd}.
\end{equation*}
We apply Lemma~\ref{lmHge} to the obtained equation for $\tilde{v}_i$ and conclude that it is uniquely solvable for each $\l\in\mathds{C}\setminus\mathds{R}$ provided $\e$ is small enough. Hence, the same is true for problem (\ref{3.21}), (\ref{5.2}). We also apply representations
(\ref{4.81}), (\ref{4.82}), (\ref{4.83}) to the problem for $\tilde{v}_i$ and this implies formulae (\ref{4.87a}), (\ref{4.87b}), (\ref{4.87c}) 
owing to the above discussed vertex conditions for $v_i^{bnd}$ and
the following identities based on a simple integration by parts:
\begin{align*}
& (g,\tilde{\psi}^{(j)})_{L_2(\g_{ex})}= -\big((\hat{\Op}_{ex}(\e)-\e^2\l)v_i^{bnd}, \tilde{\psi}^{(j)}\big)_{L_2(\g_{ex})}
\\
&\hphantom{(g,\tilde{\psi}^{(j)})_{L_2(\g_{ex})}} =\big(\Pi_{\G,M_0} \cU_{ex}'(v_i^{bnd})-2\iu\e\Th_{\G,M_0}\cU_{ex} (v_i^{bnd}),\tilde{\varPsi}^{(j)}
 \big)_{\mathds{C}^{d_0}}+\e^2\l(v_i^{bnd}, \tilde{\psi}^{(j)})_{L_2(\g_{ex})}
 \\
 &\hphantom{(g,\tilde{\psi}^{(j)})_{L_2(\g_{ex})}}=\d_{ij} + \e^2\l(v_i^{bnd}, \tilde{\psi}^{(j)})_{L_2(\g_{ex})},
\qquad i=1,\ldots,d_0,\quad j=1,\ldots,k,
\end{align*}
The proof is complete.
\end{proof}

In view of the definition of the operator $\Op_{ex}(\e)$ and the functions $v_{\g,\e}^{(i)}$ and in view of problem (\ref{3.20}), (\ref{5.2})
 we conclude that
\begin{equation}\label{3.39}
\cR_\g(\e,\l)(f_\G,f_\g)=w_\e =\e^2(\Op_{ex}(\e)-\e^2\l)^{-1}\chi_\g f_\g+ \sum\limits_{i=1}^{d_0}
b_i(\e)
v_{\g,\e}^{(i)}
\end{equation}
where the numbers $b_i$
are defined as
\begin{equation*}
\rb(\e)=
\begin{pmatrix}
b_1(\e)
\\
\vdots
\\
b_{d_0}(\e)
\end{pmatrix}:= \tilde{\Psi}^*\rL_{\exp}(\e) \big(
 \rL_{\cos}(\e)\ra'(\e)-\e\rL_{\tau}(\e) \rL_{\sin}(\e) \ra(\e)\big).
\end{equation*}

In order to rewrite identities (\ref{9.28}) to a system of linear equations, we introduce three auxiliary $d_0\times d_0$ matrices:
\begin{equation}
\begin{gathered}
\rT_\G(\e):=
\begin{pmatrix}
T_{\G,11}(\e) & \ldots & T_{\G,d_01}(\e)
\\
\vdots & & \vdots
\\
T_{\G,1d_0}(\e) & \ldots & T_{\G,d_0d_0}(\e)
\end{pmatrix},
\qquad
\rT_\g(\e):=
\begin{pmatrix}
T_{\g,11}(\e) & \ldots & T_{\g,d_01}(\e)
\\
\vdots & & \vdots
\\
T_{\g,1d_0}(\e) & \ldots & T_{\g,d_0d_0}(\e)
\end{pmatrix},
\\
\rT_\g'(\e):=
\begin{pmatrix}
T_{\g,11}'(\e) & \ldots & T_{\g,d_01}'(\e)
\\
\vdots & & \vdots
\\
T_{\g,1d_0}'(\e) & \ldots & T_{\g,d_0d_0}'(\e)
\end{pmatrix},
\end{gathered}\label{3.26}
\end{equation}
with the entries
\begin{equation}\label{3.25a}
\begin{aligned}
&T_{\G,li}(\e):=\sp_i(\e)\frac{d v_{\G,\e}^{(l)}\big|_{\ed_i}}{dx_i}(M_0) -\iu\sq_i(\e)v_{\G,\e}^{(l)}\big|_{\ed_i}(M_0), \qquad
i,l=1,\ldots,d_0,
\\
&T_{\g,li}(\e):=v_{\g,\e}^{(l)}\big|_{\ed_i^{ex}}(M_j), \qquad T_{\g,li}'(\e):=\sp_i(\e)\frac{d v_{\g,\e}^{(l)}\big|_{\ed_i^{ex}}}{d\xi_i}(M_j) -\iu\e\sq_i(\e)
v_{\g,\e}^{(l)}\big|_{\ed_i^{ex}}(M_j),
\end{aligned}
\end{equation}
where $i\in J_j$, $j=1,\ldots,n$, $l=1,\ldots,d_0$.

 In view of the definition of the matrices $\rT_\G$, $\rT_\g$, $\rT_\g'$ in (\ref{3.26}), (\ref{3.25a}), the continuity conditions (\ref{9.28})
can be rewritten as two systems of linear equations:
\begin{equation}\label{3.27-1}
\begin{aligned}
 \ra-\rT_\g
 \rb
 =& \e^2\cU_\g\big((\Op_{ex}-\e^2\l)^{-1}\chi_\g f_\g\big),
\\
\e\rT_\G\ra -\rT_\g'
\rb
=&
\e^2\Pi_{\G,M_0}\cU_\g' \big((\Op_{ex}-\e^2\l)^{-1} \chi_\g f_\g\big)
\\
&-\iu\e^3\Th_{\G,M_0} \cU_\g\big((\Op_{ex}-\e^2\l)^{-1}\chi_\g f_\g\big)
-\e\Pi_{\G,M_0}
\cU_{M_0}' \big((\Op_\G-\l)^{-1}f_\G\big)
\end{aligned}
\end{equation}
where we do not indicate explicitly the dependence of the matrices and operators on $\e$ to simplify the writing and
\begin{equation*}
\cU_\g'(u):=\left(
\frac{du\big|_{\ed_i^{ex}}}{d\xi_i}
(M_j^{ex})
\right)_{i\in J_j,\ j=1,\ldots,n}
\end{equation*}
We know apriori that the resolvent $(\Op_\e-\l)^{-1}$ is well-defined. Hence, the above system of linear equations is uniquely solvable. Our next step is to study the dependence of its solution on $\e$ and this will be done in the next subsection.

\subsection{Solution of linear system of equations}

In order to study the solution of system (\ref{3.27-1}), we shall employ certain properties of the matrices involved in this system. These properties are established in the following auxiliary lemma.

\begin{lemma}\label{lmT}
The matrices $\rT_\G(\e)$ and $\rT_\g'(\e)$ are analytic in sufficiently small $\e$,
while the matrix $\rT_\g(\e)$ is meromorphic in sufficiently small $\e$.
The leading terms of the Laurent series of the matrix $\rT_\g$
is as follows:
\begin{align}\label{3.28a}
&\rT_\g(\e)=\e^{-2}\rT_{-2} +\e^{-1}\rT_{-1} +\rT_0(\e),
\\
\label{3.28b}
&\rT_{-2}:=
\begin{pmatrix}
\tilde{\Psi}_0(\rX_0-\l\rG_0)^{-1} & 0 & 0
\end{pmatrix},
\\
&\rT_{-1}:=
\begin{pmatrix}
\Phi & \tilde{\Psi}_1 \rQ_1^{-1} & 0
\end{pmatrix} + \tilde{\Psi}_0\Ups
\label{3.28c}
\end{align}
where
$\tilde{\Psi}_0:=
\begin{pmatrix}
\tilde{\varPsi}_1 & \ldots & \tilde{\varPsi}_{k_0}
\end{pmatrix}$,
$\tilde{\Psi}_1=
\begin{pmatrix}
\tilde{\varPsi}_{k_0+1} & \ldots & \tilde{\varPsi_k}
\end{pmatrix}$, the symbols
$\Phi$ and $\Ups$ denote some matrices of sizes $d_0\times k_0$ and $k_0\times d_0$, respectively, while $\rT_0(\e)$
is an analytic in $\e$ matrix.
The identities hold:
\begin{align}
&\rT_\g'= \rL_{\exp}^{-1}
\rL_{\cos}^{-1}
\tilde{\Psi} +\e\rL_{\tau}\rL_{\sin}\rL_{\cos}^{-1} \rT_\g,\label{6.11a}
\\
&
\begin{aligned}
\Pi_{\G,M_0}\cU_\g'\big((\Op_{ex}-\e^2\l)^{-1}\chi_\g f_\g\big)
&-\iu\e\Th_{\G,M_0}\cU_\g\big((\Op_{ex}-\e^2\l)^{-1}\chi_\g f_\g\big)
\\
&=\e
\rL_{\tau}\rL_{\sin} \rL_{\cos}^{-1} \cU_\g\big((\Op_{ex}-\e^2\l)^{-1}\chi_\g f_\g\big).
\end{aligned}
\label{7.23}
\end{align}
\end{lemma}

\begin{proof}
The analyticity in $\e$ of the matrix $\rT_\G(\e)$ is implied by its definition and the analyticity in $\e$ of the functions $v_{\G,i}^{(\e)}$. The meromorphic dependence of the matrix $\rT_\g$ on $\e$ and formulae (\ref{3.28b}), (\ref{3.28c}) for the first term of its Laurent series is a direct implication of Lemma~\ref{lm4.9}. This lemma also yields that the matrix $\rT_\g'$ is meromorphic in $\e$.

On the edges $\ed_i^{ex}$, the equation for $v_{\g,\e}^{(i)}$ in (\ref{3.21}) can be solved explicitly. Taking into consideration the vertex condition in (\ref{3.21}), we find that on each edge $\ed_i^{ex}$ the function $v_{\g,\e}^{(l)}$, $l=1,\ldots,d_0$, is given by the following formula:
\begin{equation}\label{7.15}
v_{\g,\e}^{(l)}(\xi)=\Psi_{li}\phi_i'(\xi_i-1,\e) +
\frac{\phi_i(\xi-1,\e)}{\phi_i(-1,\e)} \Big(
v_{\g,\e}^{(l)}\big|_{\ed_i^{ex}}(M_j) -\Psi_{li}\phi_i'(-1,\e)
\Big),\qquad i\in J_j,\quad j=1,\ldots,n,
\end{equation}
where $\Psi_{li}$ are the components of the vectors $\tilde{\varPsi}_l$.
By straightforward calculations we then find that
\begin{align*}
\Pi_{\G,M_0} \cU_\g'(v_{\g,\e}^{(l)})-\iu\e
\Th_{\G,M_0} \cU_\g(v_{\g,\e}^{(l)})=&\rL_{\exp}^{-1}
\big(\rL_{\cos} +\rL_{\sin}^2\rL_{\cos}^{-1}
\big) \tilde{\varPsi}_l
+\e\rL_{\tau}\rL_{\sin}\rL_{\cos}^{-1}
\cU_\g(v_{\g,\e}^{(l)})
\\
=&\rL_{\exp}^{-1}\rL_{\cos}^{-1}
 \tilde{\varPsi}_l
+\e\rL_{\tau}\rL_{\sin}\rL_{\cos}^{-1}
\cU_\g(v_{\g,\e}^{(l)}).
\end{align*}
This identity implies formula (\ref{6.11a}). Then it follows from this formula and expansion (\ref{3.28a}) that the matrix $\rT_\g'(\e)$ is in fact analytic in $\e$.

Identity (\ref{7.23}) can be proved in the same way as (\ref{6.11a}) by employing the following formula similar to (\ref{7.15}):
\begin{equation*}
v_\e(\xi)=\frac{\phi_i(\xi-1,\e)}{\phi_i(-1,\e)} v_\e\big|_{\ed_i^{ex}}(M_j)\quad\text{on}\quad \ed_i^{ex},\quad i\in J_j,\quad j=1,\ldots,n,\qquad v_\e:=\big(\Op_{ex}(\e)-\e^2\l\big)^{-1}\chi_\g f_\g.
\end{equation*}
The proof is complete.
\end{proof}

Identities (\ref{6.11a}), (\ref{7.23}) allow us to simplify system (\ref{3.27-1}). Namely, we multiply the first equation by $\e\rL_{\tau}\rL_{\sin} \rL_{\cos}^{-1}$ and deduct the result from the second equation. This equation is then transformed as follows:
\begin{equation}\label{9.29}
\e\big(\rT_\G-\rL_{\tau}\rL_{\sin} \rL_{\cos}^{-1}\big)\ra -
\rL_{\exp}^{-1}
\rL_{\cos}^{-1}
\tilde{\Psi}
\rb=-\e\Pi_{\G,M_0}
\cU_{M_0}' \big((\Op_\G-\l)^{-1}f_\G\big).
\end{equation}
Since for small $\e$ we have
\begin{equation*}
\rL_{\exp}(\e)=\rI_{d_0}+O(\e),\qquad \rL_{\cos}(\e)=\rI_{d_0}+O(\e^2),\qquad \rL_{\sin}(\e)=O(\e),
\end{equation*}
we can express the vector $
\rb$ from (\ref{9.29}):
\begin{equation}\label{9.31}
\rb= \e \tilde{\Psi}^*\big(\rT_\G(0)+\e\rZ_1\big)\ra +\e \tilde{\Psi}^* \rZ_0
\cU_{M_0}' \big((\Op_\G-\l)^{-1}f_\G\big).
\end{equation}
where
 $\tilde{\Psi}:=\begin{pmatrix} \tilde{\varPsi}_1 & \ldots & \tilde{\varPsi}_{d_0}
\end{pmatrix}$, while
$\rZ_0=\rZ_0(\e)$, $\rZ_1=\rZ_1(\e)$ are analytic in $\e$ matrices and
\begin{equation}\label{9.33}
\rZ_0(\e)
=
\Pi_{\G,M_0}(0)+
O(\e).
\end{equation}
We substitute formula (\ref{9.31}) into the first equation in (\ref{3.27-1}) and employ identities (\ref{3.28a}), (\ref{3.28b}), (\ref{3.28c}). Then we multiply the resulting equation by $\e \tilde{\Psi}^*$
and this leads us to the following equation for the vector $\ra$:
\begin{equation}\label{9.34}
\e \tilde{\Psi}^*\ra- \tilde{\Psi}^*\rT_{-2} \tilde{\Psi}^* \big(\rT_\G(0)+\e \rZ_1\big)\ra-
\e \tilde{\Psi}^*\rT_{-1}\tilde{\Psi}^*\rT_\G(0) \ra + \e^2\rZ_2\ra= \e\tilde{\Psi}^*\rh,
\end{equation}
where $\rZ_2=\rZ_2(\e)$ is an analytic in $\e$ matrix and $\rh=\rh(\e)$ is a meromorphic in $\e$ vector defined as
\begin{equation}\label{9.35}
\rh(\e):=\e^2\cU_\g\big((\Op_{ex}(\e)-\e^2\l)^{-1}\chi_\g f_\g\big) + \e\rT_\g(\e)\tilde{\Psi}^*\rZ_0(\e)
\cU_{M_0}' \big((\Op_\G(\e)-\l)^{-1}f_\G\big).
\end{equation}
In order to study the obtained equation, we shall need the following simple statement.
\begin{lemma}
The matrix $\rT_\G(0)$ is invertible.
\end{lemma}
\begin{proof}
If the matrix $\rT_\G(0)$ degenerates, then its columns are linearly dependent with some coefficients $C_i$, $i=1,\ldots,d_0$.
In view of formula (\ref{3.25a}) for $T_\G^{(ip)}$ and problems (\ref{3.15}), (\ref{2.9}), we see immediately that then the corresponding linear combination $v:=\sum\limits_{i=1}^{d_0} C_i v_{\G,i}^{(0)}$ also solves problem (\ref{3.15}), (\ref{2.9}) but satisfies the homogeneous Neumann condition at $M_0$. Since the parameter $\l$ is non-real, it can not be an eigenvalue of a self-adjoint operator on $\G$ with differential expression $\hat{\Op}(0)$
subject to vertex conditions (\ref{2.34}) and to the Neumann condition at $M_0$. Hence, the function $v$ vanishes identically and in view of boundary conditions (\ref{3.16}) we see that $C_i=0$, $i=1,\ldots,d_0$.
The proof is complete.
\end{proof}

The proven lemma implies that the matrix $\rT_\G(0)+\e\rZ_1(\e)$ is invertible and the inverse is analytic in $\e$. Then denoting
\begin{equation}\label{9.36}
\tilde{\ra}(\e):=\tilde{\Psi}^*
\big(\rT_\G(0)+\e\rZ_1(\e)\big)\ra(\e),
\end{equation}
we rewrite equation (\ref{9.34}) to
\begin{equation}\label{9.37}
\left(- \tilde{\Psi}^*\rT_{-2}-
\e \tilde{\Psi}^*\rT_{-1} +\e \tilde{\Psi}^*\rT_\G^{-1}(0)\tilde{\Psi} + \e^2\rZ_3\right)\tilde{\ra}= \e\tilde{\Psi}^*\rh,
\end{equation}
where $\rZ_3=\rZ_3(\e)$ is an analytic in $\e$ matrix. According formulae (\ref{3.28b}), (\ref{3.28c}), the matrices $\tilde{\Psi}^*\rT_{-2}$ and $\tilde{\Psi}^*\rT_{-1}$ are of the form
\begin{equation}\label{9.38}
\tilde{\Psi}^*\rT_{-2}=
\begin{pmatrix}
(\rX_0-\l_0\rG_0)^{-1} & 0
\\
0 & 0
\end{pmatrix},\qquad \tilde{\Psi}^*\rT_{-1}=
\begin{pmatrix}
\Phi_1 & 0
\\
\Phi_2 & \rQ_2
\end{pmatrix},\qquad \rQ_2:=
\begin{pmatrix}
\rQ_1^{-1} & 0
\\
0 & 0
\end{pmatrix}
\end{equation}
where the widths and heights of the blocks in $\tilde{\Psi}^*\rT_{-2}$ and $\tilde{\Psi}^*\rT_{-1}$ are $k_0$, $d_0-k_0$, and $\Phi_1$, $\Phi_2$ are some fixed matrices of sizes respectively $k_0\times k_0$, $(d_0-k_0)\times k_0$.
The widths and heights of the blocks in $\rQ_2$ are $k$, $d_0-k$.

In view of the above shown structure of the leading terms in the matrix of equation (\ref{9.37}), it is natural to seek its solution  as
\begin{equation}\label{9.39}
\tilde{\ra}(\e)=
\begin{pmatrix}
\tilde{\ra}_0(\e)
\\
\tilde{\ra}_1(\e)
\end{pmatrix},
\end{equation}
where $\tilde{\ra}_0$ is a vector of height $k_0$, while $\tilde{\ra}_1$ is a vector of height $k-k_0$. We also introduce two simple projectors on $\mathds{C}^{d_0}$
acting as

\begin{equation}\label{3.54}
\rR_0
\begin{pmatrix}
a_1
\\
\vdots
\\
a_{d_0}
\end{pmatrix}:=
\begin{pmatrix}
a_1
\\
\vdots
\\
a_{k_0}
\end{pmatrix},
\qquad
\rR_1
\begin{pmatrix}
a_1
\\
\vdots
\\
a_{d_0}
\end{pmatrix}
:=
\begin{pmatrix}
a_{k_0+1}
\\
\vdots
\\
a_{d_0}
\end{pmatrix}.
\end{equation}
It is obvious that $\tilde{\ra}_0=\rR_0 \tilde{\ra}$,
$\tilde{\ra}_1=\rR_1 \tilde{\ra}$.

We substitute representations (\ref{9.39}) into equation (\ref{9.37}) and apply then the projectors $\rR_0$ and $\rR_1$ to this equation. Then, in view of identities (\ref{9.38}), we immediately get:
\begin{equation}\label{9.40}
\begin{aligned}
&\big((\rX_0-\l_0\rG_0)^{-1}+\e\rZ_4\big)\tilde{\ra}_0 + \e\rZ_5 \tilde{\ra}_1(\e)=\e\rR_0 \tilde{\Psi}^*\rh,
\\
&\rQ_3 \tilde{\ra}_1(\e)+\rZ_6 \tilde{\ra}_0=\rR_1 \tilde{\Psi}^*\rh,\qquad \rQ_3:=\rR_1 \tilde{\Psi}^*\rT_\G^{-1}(0)\tilde{\Psi}\rR_1-\rQ_2,
\end{aligned}
\end{equation}
where $\rZ_4=\rZ_4(\e)$, $\rZ_5=\rZ_5(\e)$, $\rZ_6=\rZ_6(\e)$ are some analytic in $\e$ matrices of sizes respectively $k_0\times k_0$, $k_0\times (d-k_0)$, $(d-k_0)\times k_0$.

Since the matrix $(\rX_0-\l\rG_0)^{-1}$ is invertible, we can immediately solve the first equation in (\ref{9.40}):
\begin{equation}\label{9.41}
\tilde{\ra}_0(\e)=\e\rZ_8(\e)
\rR_0 \tilde{\Psi}^*\rh(\e) -\e\rZ_7(\e) \tilde{\ra}_1(\e),\qquad \rZ_8(0)=\rX_0-\l\rG_0,
\end{equation}
where $\rZ_7=\rZ_7(\e)$, $\rZ_8=\rZ_8(\e)$ are some analytic in $\e$ matrices of sizes $k_0\times(d_0-k_0)$.
We substitute the obtained identity into the second equation in (\ref{9.40}) and we get:
\begin{equation}\label{9.42}
(\rQ_3 +\e\rZ_9)\tilde{\ra}_1=\rR_1 \tilde{\Psi}^*\rh+
\e\rZ_{10} \rR_0 \tilde{\Psi}^*\rh,
\end{equation}
where $\rZ_9=\rZ_9(\e)$, $\rZ_{10}=\rZ_{10}(\e)$ are some analytic in $\e$ matrices. In order to solve the obtained equation, we shall employ the following simple lemma.

\begin{lemma}\label{lmQ3}
The matrix $\rQ_3$ is invertible.
\end{lemma}
\begin{proof}
We let $\e=0$ in (\ref{3.15}), multiply then the equation
by $\overline{v_{\G,j}^{(0)}}$ and integrate twice by parts over the graph $\G$ taking into consideration the vertex conditions for $v_{\G,i}^{(0)}$ and $v_{\G,j}^{(0)}$:
\begin{align*}
&0=\big((\hat{\Op}(0)-\l)v_{\G,i}^{(0)},v_{\G,j}^{(0)}\big)_{L_2(\G)}
\\
&=\big(\Pi_{\G,M_0}(0)\cU_{M_0}'(v_{\G,i}^{(0)})
-\iu\Th_{\G,M_0}(0)\cU_{M_0}(v_{\G,i}^{(0)}), \cU_{M_0}(v_{\G,j}^{(0)})\big)_{\mathds{C}^{d_0}}
\\
&\hphantom{=}- \big(\cU_{M_0}(v_{\G,i}^{(0)}),\Pi_{\G,M_0}(0)\cU_{M_0}'(v_{\G,j}^{(0)}) -\iu\Th_{\G,M_0}(0)\cU_{M_0}(v_{\G,j}^{(0)})\big)_{\mathds{C}^{d_0}}+
\big(v_{\G,i}^{(0)},(\hat{\Op}(0) -\overline{\l})v_{\G,j}^{(0)}\big)_{L_2(\G)}.
\end{align*}
Since
\begin{equation*}
\big(v_{\G,i}^{(0)},(\hat{\Op}(0) -\overline{\l})v_{\G,j}^{(0)}\big)_{L_2(\G)}=
-2\iu \IM\l\,
\big(v_{\G,i}^{(0)},v_{\G,j}^{(0)}\big)_{L_2(\G)},
\end{equation*}
by (\ref{3.16}), (\ref{3.25a}) we obtain:
\begin{equation*}
T_\G^{(ji)}(0)-\iu\IM\l \big(v_{\G,i}^{(0)},v_{\G,j}^{(\e)}\big)_{L_2(\G)} =\overline{T_\G^{(ij)}(0)-\iu\IM\l \big(v_{\G,i}^{(0)},v_{\G,j}^{(0)}\big)_{L_2(\G)}}.
\end{equation*}
Hence, the matrix $\rT_\G(0)-\iu\IM\l\, \rG_\G$ is self-adjoint, where $\rG_\G$ is the Gram matrix of the functions $v_{\G,i}^{(0)}$.
We also observe that the matrix $\rG_\G$ is positive definite.

Let $\rc\in\rR_1\tilde{\Psi}^*\mathds{C}^{d_0}$ be a vector such that $\rQ_3\rc=0$.
Then, in view of the definition of the matrix $\rQ_3$ in (\ref{9.40}),
\begin{align*}
\big(\rR_1\rQ_2\rR_1\rc,\rc\big)_{\mathds{C}^{d_0}}= &\big(\rR_1\tilde{\Psi}^*\rT_\G^{-1}(0) \tilde{\Psi}\rR_1\rc,\rc\big)_{\mathds{C}^{d_0}} = \big( \rT_\G(0)\tilde{\Psi}\rR_1 \rc,\tilde{\Psi}\rR_1 \rc\big)_{\mathds{C}^{d_0}}
\\
=&\big(\tilde{c}, \rT_\G(0)\tilde{\rc}\big)_{\mathds{C}^{d_0}} = \big(\tilde{c}, (\rT_\G(0)- \iu\IM\l\rG_\G)\tilde{\rc}\big)_{\mathds{C}^{d_0}} + \iu\IM\l\big(\tilde{c}, \rG_\G \tilde{\rc}\big)_{\mathds{C}^{d_0}},
\end{align*}
where
$\tilde{\rc}:=\rT_\G^{-1}(0)\tilde{\Psi}\rR_1\rc$, $\rR_1\rc=\tilde{\Psi}\rT_\G(0)\tilde{\rc}$. Since both matrices $\rQ_2$ and $(\rT_\G(0)-\iu\IM\l\rG_\G)$ are self-adjoint, $\IM\l\ne0$ and $\rG_\G$ is positive definite, we above identity implies immediately that
$\tilde{\rc}=0$ and $\rR_1\rc=0$. This completes the proof.
\end{proof}

The proven lemma allows us to solve equation (\ref{9.42}):
\begin{equation}\label{9.43}
\tilde{\ra}_1(\e)=\rQ_3^{-1}\rR_1 \tilde{\Psi}^*\rh(\e)+
\e\rZ_{11}(\e) \rR_0 \tilde{\Psi}^*\rh(\e),
\end{equation}
where $\rZ_{11}=\rZ_{11}(\e)$ is some analytic in $\e$ matrix. By this formula and (\ref{9.41}) we find the vector $\tilde{\ra}(\e)$ by (\ref{9.39}) and we recover then the vector $\ra(\e)$ by (\ref{9.36}). This determines vector $\rb(\e)$ by (\ref{9.31}) and then we can find functions $W_\e$ and $w_\e$ by (\ref{3.17}), (\ref{3.39}). In order to establish the analyticity of these functions, we first need to clarify the behavior of the vector $\tilde{\Psi}^*\rh(\e)$ as $\e\to+0$. Definition (\ref{9.35}) of the vector $\rh(\e)$ involves the matrix $\rT_\g(\e)$, which has a second order pole at $\e=0$ and this is why apriori we can only say that the vector $\rh(\e)$ is meromoprhic at $\e=0$. However, when we recover the vectors $\ra(\e)$ and $\rb(\e)$, this pole is eliminated and these vectors turn out to be analytic in $\e$. Namely, the following lemma holds true.

\begin{lemma}\label{lm-ab}
The vectors $\ra(\e)$ and $\rb(\e)$ are analytic in $\e$ as linear bounded functionals on $L_2(\G)\oplus L_2(\g)$. The identities hold true:
\begin{align}\label{9.44}
&
\begin{aligned}
\begin{pmatrix}
\tilde{\Psi}_0 & \tilde{\Psi}_1
\end{pmatrix}^* \Big(\rT_\G(0) \ra(0)&+\Pi_{\G,M_0}(0) \cU_{M_0}' \big((\Op_\G(0)-\l)^{-1}f_\G\big)
\Big)
\\
&=\diag\big\{\l_1(\rQ),\ldots,\l_k(\rQ)\big\}
\begin{pmatrix}
\tilde{\Psi}_0 & \tilde{\Psi}_1
\end{pmatrix}^*\ra(0),
\end{aligned}
\\
&\rb(0)=0, \qquad \rR_0 \frac{d\rb}{d\e}(0)=0.
\label{9.46}
\end{align}
\end{lemma}

\begin{proof}
It follows from the definition of the vector $\rh(\e)$ in (\ref{9.35}), Lemmata~\ref{lmHge},~\ref{lmT} and identities (\ref{9.33}), (\ref{9.38}) that the vector $\rh$ is meromorphic in $\e$ and has a first order pole at $\e=0$:
\begin{equation}\label{9.45}
\begin{aligned}
\tilde{\Psi}^*\rh(\e)=&
\begin{pmatrix}
\e^{-1}(\rX_0-\l\rG_0)^{-1}+\Phi_1 & 0
\\
\Phi_2 & \rQ_2
\end{pmatrix} \tilde{\Psi}^* \Big(\Pi_{\G,M_0}(0) \cU_{M_0}'\big((\Op_\G(0)-\l)^{-1}f_\G\big)+O(\e)\Big)
\\
&+
\begin{pmatrix}
(\rX_0-\l\rG_0)^{-1}
\\
0
\end{pmatrix}
\begin{pmatrix}
(f,\tilde{\phi}^{(1)})_{L_2(\g)}
\\
\vdots
\\
(f,\tilde{\phi}^{(k_0)})_{L_2(\g)}
\end{pmatrix}+O(\e).
\end{aligned}
\end{equation}
We substitute this identity into (\ref{9.43}) and in view of the definition of the projectors $\rR_0$ and $\rR_1$ we see that the projector $\rR_1$ eliminates the pole and the vector $\ra_1(\e)$ is analytic.
Then we can find the vector $\tilde{\ra}_0$ by formula (\ref{9.41}):
\begin{equation}\label{9.47}
\tilde{\ra}_0(\e)=\rR_0 \tilde{\Psi}^* \rT_\G(0)\ra(0)= \tilde{\Psi}_0^* \Pi_{\G,M_0}(0) \cU_{M_0}'\big((\Op_\G(0)-\l)^{-1}f_\G\big)+O(\e).
\end{equation}
Now formula (\ref{9.43}) implies that the vector $\tilde{\ra}_1$ is also analytic in $\e$. Hence, it follows from (\ref{9.36}), (\ref{9.31}) that the vectors $\ra$ and $\rb$ are analytic in $\e$ as well and the first identity (\ref{9.46}) holds true. It is clear that $\ra(\e)$ and $\rb(\e)$ are bounded linear functionals on $L_2(\G)\oplus L_2(\g)$. We have just proved that these functionals are analytic on each elements of the mentioned space and therefore, they are also analytic in the norm sense.

In order to prove identity (\ref{9.44}), we introduce one more projector in $\mathds{C}^{d_0}$ similar to (\ref{3.54}):
\begin{equation*}
\rR_2
\begin{pmatrix}
a_1
\\
\vdots
\\
a_{d_0}
\end{pmatrix}=
\begin{pmatrix}
a_1
\\
\vdots
\\
a_k
\end{pmatrix},
\end{equation*}
and we apply this projector to (\ref{9.45}) and to equation (\ref{9.37}). Employing then identities (\ref{9.47}) and (\ref{9.36}), we get:
\begin{equation*}
-\rQ_1^{-1}\rR_2 \tilde{\Psi}^*\rT_\G(0)
\ra(0) + \rR_2 \tilde{\Psi}^*
\ra(0)=\rQ_1^{-1}\rR_2 \tilde{\Psi}^* \Pi_{\G,M_0}(0) \cU_{M_0}'\big((\Op_\G(0)-\l)^{-1}f_\G\big).
\end{equation*}
The above identity, (\ref{9.36}) with $\e=0$ and (\ref{9.47}) imply (\ref{9.44}). This identity and (\ref{9.31}), (\ref{9.33}) yield immediately the second identity in (\ref{9.46}). The proof is complete.
 \end{proof}

The proven lemma, formula (\ref{3.17}) and the analyticity of the resolvent $(\Op_\G(\e)-\l)^{-1}$ and the functions $v_{\G,\e}^{(i)}$ imply that the operator $\cR_\G(\e,\l)$ is analytic in $\e$ as a bounded operator from $L_2(\G)\oplus L_2(\g)$ into $\dH^2(\G)$. Employing identities (\ref{9.44}), (\ref{9.16}), (\ref{4.30}), (\ref{4.31}) and vertex conditions (\ref{3.16}), it is straightforward to confirm that the function $\cR_\G(0,\l)f_\G$ satisfies vertex conditions (\ref{2.34}), (\ref{2.32}). This proves the first identity in (\ref{2.6}).

In view of Lemma~\ref{lmHge}, the first term in the right hand side in (\ref{3.39}), the operator $\e^2(\Op_{ex}(\e)-\e^2\l)^{-1}\chi_\g f_\g$, is analytic in $\e$ as acting from $L_2(\g)$ into $\dH^2(\g)$. Owing to identities (\ref{9.46}) and Lemma~\ref{lm4.9}, each function $b_i(\e) v_{\g,\e}^{(i)}$, $i=1,\ldots,d_0$, is also analytic in $\e$ since the negative powers of $\e$ coming from $v_{\g,\e}^{(i)}$ are compensated by the positive powers in $b_i(\e)$. Hence, the sum in (\ref{3.39}) is also analytic as an operator from $L_2(\g)\oplus L_2(\G)$ into $\dH^2(\g)$. Therefore, the same is true for operator $\cR_\g(\e,\l)$ given by (\ref{3.39}). It also follows from identities (\ref{9.46}) and Lemma~\ref{lm4.9} that
$\cR_\g(0,\l)=\sum\limits_{i=1}^{k} C_i \psi^{(i)}$,
 where $C_i$ are some constants and in view of definition (\ref{3.26}), (\ref{3.25a}) of the matrix $\rT_\g$ and the first equation in (\ref{3.27-1}) we conclude immediately that $C_i=a_i(0)$, $i=1,\ldots,d_0$. Since by (\ref{3.17}) and by the first identity in (\ref{2.6}) we also have $\ra(0)=\Psi^*\cU_{M_0}\big((\Op_0-\l)^{-1}f_\G\big)$, we get the second identity in (\ref{2.6}).

Owing to the embeddings $\dH^2(\G)\subset \dC^1(\G)$ and $\dH^2(\g)\subset \dC^1(\g)$, the above proven analyticity in $\e$ of the operators $\cR_\G(\e,\l)$ and $\cR_\g(\e,\l)$ imply that they also bounded and analytic in $\e$ as acting into $\dC^1(\G)$ and $\dC^1(\g)$. We can express the second derivative of the functions $\cR_\G(\e,\l)(f_\G,f_\g)$ and $\cR_\g(\e,\l)(f_\G,f_\g)$ from their differential equations, see (\ref{3.13}), (\ref{3.20}), via these functions and their first derivatives. In view of the established analyticity in $\e$ of these functions in $\dC^1(\G)$ and $\dC^1(\g)$ and the assumed smoothness and analyticity in $\e$ of the functions $p_\G$, $q_\G$, $V_\G$ and $p_\g$, $q_\g$, $V_\g$, we conclude that the functions $\cR_\G(\e,\l)(f_\G,f_\g)$ and $\cR_\g(\e,\l)(f_\G,f_\g)$ are also analytic in the norms of the spaces $\dC^2(\G)$ and $\dC^2(\g)$. Therefore, the operators $\cR_\G(\e,\l)$ and $\cR_\g(\e,\l)$ are bounded and analytic in $\e$ as acting from $L_2(\G)\oplus L_2(\g)$ into $\dC^2(\G)$ and $\dC^2(\g)$.

\section*{Acknowledgments}

The author thanks a referee for valuable and useful remarks, which allowed to improve significantly the initial version of the paper.

\smallskip

\noindent The research is supported by the Russian Science Foundation
(grant no. 20-11-19995).

\end{document}